\documentclass[11pt,english,oneside]{amsart}

\usepackage[breaklinks,colorlinks,linkcolor=blue,citecolor=blue,urlcolor=black]{hyperref}
\usepackage[T1]{fontenc}
\allowdisplaybreaks[4]
\usepackage{geometry}
\usepackage{indentfirst}
\usepackage{amssymb}
\usepackage{color}
\usepackage{graphicx}
\usepackage{subfigure}
\usepackage{enumerate}
\usepackage{float}

\usepackage{booktabs}
\usepackage{array, caption, threeparttable}
\usepackage[font=small,labelfont=bf,labelsep=none]{caption}
\usepackage{graphicx,color}
\usepackage{amsmath, amssymb, graphics}
\usepackage{graphicx}
\usepackage{epstopdf}

\numberwithin{figure}{section}

\makeatletter
\numberwithin{equation}{section} 
\numberwithin{figure}{section} 
\@ifundefined{theoremstyle}{\usepackage{amsthm}}{}
\theoremstyle{plain}
\newtheorem{thm}{Theorem}[section]
\newtheorem{cor}[thm]{Corollary}
\newtheorem{prop}[thm]{Proposition}
\newtheorem{rmk}[thm]{Remark}

\newtheorem{lem}[thm]{Lemma}
\newtheorem{defn}[thm]{Definition}

\usepackage{geometry}

\newcommand{\R}{\mathbb{R}}

\def\exp{\hbox{\rm exp}}
\smallskip
\def\<{{\langle }}
\def\>{{\rangle }}

\makeatother

\usepackage{babel}

\def\exp{\hbox{\rm exp}}
\smallskip
\def\<{{\langle }}
\def\>{{\rangle }}

\makeatother

\usepackage[mathscr]{eucal}
\usepackage{amssymb}
\usepackage{latexsym}
\usepackage{amsthm}
\theoremstyle{plain}
\newtheorem{theorem}{Theorem}[section]

\makeatletter
\@addtoreset{equation}{section}

\usepackage{color}

\usepackage{graphicx,color}

\title[Examples of compact embedded mean convex $\lambda$-hypersurfaces]
{Examples of compact embedded mean convex $\lambda$-hypersurfaces}
\author{Qing-Ming Cheng, Junqi Lai and  Guoxin Wei}

\address{Qing-Ming Cheng \newline \indent Mathematical Science Research Center, \newline
\indent Chongqing University of Technology, Chongqing 400054,  P. R. China  \newline \indent
qingmingcheng@yahoo.com}
\address{Junqi Lai \\  School of Mathematical Sciences, South China Normal University,
510631, Guangzhou,  China, 2019021668@m.scnu.edu.cn}
\address{Guoxin Wei \\  School of Mathematical Sciences, South China Normal University,
510631, Guangzhou,  China, weiguoxin@tsinghua.org.cn}

\begin{document}

	\maketitle
	
	\begin{abstract}
       There is a well-known conjecture asserts that the round sphere should be the only compact embedded self-shrinker (i.e. $0$-hypersurface) which is diffeomorphic to a sphere. S. Brendle \cite{B} confirmed the conjecture for 2-dimensional $0$-hypersurfaces. For any dimensional $\lambda$-hypersurfaces, if $\lambda<0$, we \cite{CLW} constructed compact {\it convex} embedded $\lambda$-hypersurface which is diffeomorphic to a sphere and is not a round sphere.
	   In this paper, for $\lambda>0$, we construct a compact {\it mean convex} embedded $\lambda$-hypersurface which is diffeomorphic to a sphere and is not a round sphere. In fact, for $\lambda>0$, there are no compact {\it convex} embedded $\lambda$-hypersurfaces which are diffeomorphic to spheres except a round sphere.
	\end{abstract}
	
	
	\section{Introduction}

A hypersurface $\Sigma^n \subset \R^{n+1}$ is called a $\lambda$-hypersurface if it satisfies
\begin{equation}\label{0316eq1.1}
     H + \< X, \nu\> = \lambda,
\end{equation}
where $\lambda$ is a constant, $X$ is the position vector, $\nu$ is an inward unit normal vector and $H$ is the mean curvature. 
The notion of $\lambda$-hypersurfaces was first introduced by Cheng and Wei in \cite{CW} (also see \cite{MR}).
Cheng and Wei \cite{CW} proved that $\lambda$-hypersurfaces are critical points of the weighted area functional with respect to weighted volume-preserving variations.
This equation of $\lambda$-hypersurfaces also arises in the study of isoperimetric problems in weighted (Gaussian) Euclidean spaces, which is a long-standing topic studied in various fields of science.
$\lambda$-hypersurfaces can also be viewed as stationary solutions to the isoperimetric problem in the Gaussian space.
For more information on $\lambda$-hypersurfaces, one can see \cite{CW} and \cite{MR}.

Note that the standard sphere with inward unit normal vector has positive mean curvature in our convention.
Since an orientable hypersurface has two directions of the unit normal vector,  if one changes the direction of the normal vector, the $\lambda$ will change its sign.
For example, if $\Sigma$ is a $\lambda$-hypersurface with unit normal vector $\nu$ and $\lambda=5$, then it is also a $\lambda$-hypersurface with unit normal vector $-\nu$ and $\lambda=-5$.

Firstly, we give some examples of $\lambda$-hypersurfaces. It  is well known that  there are several special complete embedded solutions to (\ref{0316eq1.1}): hyperplanes, the sphere centered at the origin,
 cylinders with an axis through the origin. Cheng and Wei \cite{CW1} constructed the first nontrivial example of a $\lambda$-hypersurface that is diffeomorphic to $\mathbb S^{n-1} \times \mathbb S^1$ using techniques similar to Angenent \cite{A}. Using a similar method to McGrath \cite{M}, Ross \cite{R} constructed a $\lambda$-hypersurface in $\R^{2n+2}$ which is diffeomorphic to $\mathbb S^{n} \times \mathbb S^{n} \times \mathbb S^1$ and exhibits an $SO(n) \times SO(n) $ rotational symmetry. Li and Wei \cite{LW} constructed an immersed $S^n$ $\lambda$-hypersurface using a similar method to \cite{D}. It is quite interesting to find other nontrivial examples of $\lambda$-hypersurfaces (also see \cite{I}).

Secondly, we introduce some rigidity results about $\lambda$-hypersurfaces.
If $\lambda=0$, $\langle X, \nu\rangle +H=\lambda=0$, then $X:\Sigma^n\to  \mathbb{R}^{n+1}$ is a self-shrinker of  mean curvature flow, which plays an important role in the study on singularities of the mean curvature flow.
 Huisken \cite{H} proved that any compact embedded, and {\it mean convex} (i.e. $H\geq0$) self-shrinker is a sphere.
Later, Colding and Minicozzi \cite{CM} generalized Huisken's results to the case of complete self-shrinkers.


For $\lambda=0$ (that is, self-shrinkers), there is a well-known conjecture that asserts the round sphere  should be the only embedded self-shrinker which is diffeomorphic to a sphere.
Brendle \cite{B} proved the above conjecture for $2$-dimensional self-shrinker.
For the higher dimensional  self-shrinkers, the conjecture is still open.
For $\lambda<0$, Cheng, Lai and Wei \cite{CLW} constructed a compact embedded $\lambda$-hypersurface that is diffeomorphic to a sphere and is not isometric to a round sphere.
In fact,
\begin{thm}\label{0902thm1.1}
For  $n \ge 2$ and $ -\frac{2}{\sqrt{n+2}} < \lambda <0$,  there exists an  embedded convex $\lambda$-hypersurface $\Sigma^n \subset \R^{n+1}$ which is diffeomorphic to $\mathbb{S}^n$ and is not  isometric to a standard sphere.
\end{thm}

 For $\lambda>0$, in this paper, we construct a compact embedded $\lambda$-hypersurface which is diffeomorphic to a sphere and is not isometric to a round sphere (see Theorem \ref{0316thm1.1}).

On the other hand, in \cite{S}, Sun developed the compactness theorem for $\lambda$-surfaces in $\mathbb R^{3}$ with uniform $\lambda$ and genus.
As an application of the compactness theorem, he also showed a rigidity theorem for convex $\lambda$-surfaces.
In the same paper, he \cite{S} proposed the problem of constructing a compact convex $\lambda$-surface which is not a sphere (see Question 4.0.4. on page 25, \cite{S}).
In \cite{CLW}, Cheng, Lai and Wei solved Sun's problem for certain $\lambda$ as the special case of our result, that is, $n=2$. In \cite{Hei}, Heilman proved that convex $n$-dimensional  $\lambda$-hypersurfaces are generalized cylinders if $\lambda>0$, which generalized the rigidity result of Colding and Minicozzi \cite{CM} to $\lambda$-hypersurfaces with $\lambda>0$.
From Heilman's result, we know that it is impossible to construct a compact convex $\lambda$-surface which is not a sphere for $\lambda>0$.

In this paper, motivated by the Hopf-type theorem, Sun's problem and \cite{D,DK,LW,S}, for $\lambda>0$, we construct nontrivial embedded mean-convex $\lambda$-hypersurfaces which are diffeomorphic to $\mathbb S^{n}$ and not isometric to a standard sphere.
In fact, we obtain the following theorem:

\begin{theorem}\label{0316thm1.1}
Given $n \ge 2$ and $\lambda\ge4\sqrt{\frac{n-1}{5}}$, there exists an  compact embedded and mean convex $\lambda$-hypersurface $\Sigma \subset \R^{n+1}$ which is diffeomorphic to $\mathbb{S}^{n} $ and is not  isometric to a standard sphere.
\end{theorem}

\begin{rmk}
  In \cite{G}, by using the explicit expressions of the derivatives of the principal curvatures at the non-umbilical points of the surface, Guang obtained that any strictly mean convex $2$-dimensional $\lambda$-hypersurface is convex if $\lambda\leq0$.
  Later, by using the maximum principle, Lee \cite{L} showed that any compact embedded and mean convex $n$-dimensional $\lambda$-hypersurface is convex if $\lambda\leq0$.
  But for $\lambda>0$, we construct $\lambda$-hypersurfaces which are  mean convex and non-convex. These examples may carry meaningful information in probability theory $($also see \cite{S}$)$.

\end{rmk}

\begin{rmk}
For $-\frac{2}{\sqrt{3}}<\lambda<0$, Chang \cite{C} proved that there exists a compact embedded $\lambda$-curve with $2$-symmetry in $\R^{2}$, which is not a circle.
\end{rmk}

The remainder of this paper is organized as follows.
In section 2, we set up the mathematical framework by reducing the $\lambda$-hypersurface equation to a nonlinear system of ordinary differential equations \eqref{0316eq2.5} for the profile curves.

In section 3, we conduct a rigorous qualitative analysis of the local solutions to this ODE system. We classify the local geometric shapes of these curves into three distinct types and determine their maximal existence intervals and boundary behaviors.

In section 4, we introduce a global classification scheme to track the turning points and piece the local graphical solutions together. Crucially, we also establish the continuous dependence of these global classes on the initial parameters, which is a fundamental prerequisite for our shooting argument.

Sections 5 and 6 are devoted to analyzing two specific families of profile curves: those emanating from the $x$-axis ($\bar{\gamma}_b$) and those emanating from the $r$-axis ($\gamma_\delta$), respectively. By tracking the curves through their first two graphical segments, we determine their precise initial shape classifications depending on the initial parameters.

Finally, in section 7, we synthesize the analytical results from the previous sections to execute a continuity (shooting) argument. By continuously varying the initial parameter, we prove the existence of a specific profile curve, thereby completing the proof of our main result, theorem \ref{0316thm1.1}.

	\section{Preliminaries}\label{prel}
	
	Let $\mathbb{H}$ denote the open half plane, namely, $\mathbb{H} = \left\lbrace (x,r)\in\R^2: x\in\R, r > 0\right\rbrace$.
	Let $SO(n)$ denote the special orthogonal group and act on $\R^{n+1} = \left\lbrace (x,y): x\in\R, y\in\R^n\right\rbrace $ in the usual way and then we can identify the space of orbits $\R^{n+1}/SO(n)$ with the closed half plane $\overline{\mathbb{H}} = \left\lbrace (x,r)\in\R^2: x\in\R, r \ge 0\right\rbrace$ under the projection(see \cite{R})
	\[
	\Pi(x,y) = (x,|y|) = (x,r).
	\]
	If a hypersurface $\Sigma$ is invariant under the action of $SO(n)$, then the projection $\Pi(\Sigma)$ will give us a profile curve in the half plane, which can be parametrized by Euclidean arc length and written as $\gamma(s) = (x(s), r(s))$.
	Conversely, if we have a curve $\gamma(s) = (x(s), r(s)),\ \ s \in (a, b)$ parametrized by Euclidean arc length in $\overline{\mathbb{H}}$, then we can reconstruct the hypersurface by
	\begin{gather}
	{ X : (a, b) \times S^{n-1}(1) \to \R^{n+1} }, \notag \\
	(s,\alpha) \mapsto (x(s), r(s)\alpha). \label{0316eq2.1}
	\end{gather}
	Let
	\begin{equation}\label{0316eq2.2}
	\nu = (-\dot{r}, \dot{x}\,\alpha),
	\end{equation}
	where the dot the denotes taking derivative with respect to arc length $s$. A direct calculation shows that $\nu$ is a unit normal vector for the hypersurface.
	Then we can calculate that the principal curvatures of the hypersurface(see \cite{CW1,DD}):
	\begin{equation}\label{0426eq2.3}
	\begin{aligned}
	\kappa_i & = - \frac{\dot{x}}{r}, \ \ \ \ i=1,\,2,\,\dots,\,n-1, \\
	\kappa_n &= \dot{x}\,\ddot{r}-\ddot{x}\,\dot{r}.
	\end{aligned}
	\end{equation}
	Hence the mean curvature equals to
	\begin{equation}\label{0316eq2.3}
	H =  \dot{x}\,\ddot{r}-\ddot{x}\,\dot{r} - (n-1)\frac{\dot{x}}{r},
	\end{equation}
	and then, by (\ref{0316eq2.1}), (\ref{0316eq2.2}) and (\ref{0316eq2.3}), equation (\ref{0316eq1.1}) reduces to(see also \cite{CW1,DK,R})
	\begin{equation}\label{0316eq2.4}
	\dot{x}\,\ddot{r} - \ddot{x}\,\dot{r} = (\frac{n-1}{r} - r )\dot{x} +   x\,\dot{r} + \lambda,
	\end{equation}
	where $(\dot{x})^2 + (\dot{r})^2 = 1$.
	Let $\theta(s)$ denote the angle between the tangent vector of the profile curve and the $x$-axis, then (\ref{0316eq2.4}) can be written as the following system of differential equations:
	\begin{equation}\label{0316eq2.5}
	\left\lbrace
	\begin{aligned}
	\dot{x} &= \cos\theta,\\
	\dot{r} &= \sin\theta, \\
	\dot{\theta} &=  (\frac{n-1}{r} - r )\cos\theta +   x\,\sin\theta + \lambda.
	\end{aligned}
	\right.
	\end{equation}
	Let $P$ denote the projection from $\overline{\mathbb{H}} \times \R$ to $\overline{\mathbb{H}}$. Obviously, the curve $\gamma(s)= P(\Gamma(s))$ generates a $\lambda$-hypersurface via (\ref{0316eq2.1}) provided $\Gamma(s)$ is a solution to (\ref{0316eq2.5}).

	Letting $(x_0, r_0, \theta_0)$ be a point in $\mathbb{H}\times \mathbb{R}$, by an existence and uniqueness theorem of the solutions for first order ordinary differential equations, there is a unique solution $\Gamma(x_0, r_0, \theta_0)(s)$ to (\ref{0316eq2.5}) satisfying initial conditions $\Gamma(x_0, r_0, \theta_0)(0) = (x_0, r_0, \theta_0) $.
	Moreover, the solution depends smoothly on the initial conditions.
	Note that (\ref{0316eq2.5}) has a singularity at $r = 0$ due to the term $\frac{1}{r}$.
	But there is also an existence and uniqueness theorem for (\ref{0316eq2.5})  with initial conditions at singularity and $\theta_0 = \pi/2$ (see lemma \ref{lem2.2})
	In other words, there is a unique solution $\Gamma(b, 0, \pi/2)(s)$ to (\ref{0316eq2.5}) satisfying initial conditions $\Gamma(b, 0, \pi/2)(0) = (b, 0, \pi/2) $.

	For convenience, we henceforth denote $\Gamma(0, \delta, 0)(s)$, $P(\Gamma(0, \delta, 0)(s))$, $\Gamma(b, 0, \pi/2)(s)$ and $P(\Gamma(b, 0, \pi/2)(s))$ by $\Gamma_\delta(s)$, $\gamma_\delta(s)$, $\overline\Gamma_b(s)$ and $\overline\gamma_b(s)$ respectively.
	We also assume $\Gamma_\delta(s) = (x_\delta(s), r_\delta(s), \theta_\delta(s))$ and $\overline{\Gamma}_b(s) = (\bar{x}_b(s), \bar{r}_b(s), \bar{\theta}_b(s))$.

	When $\dot{x} = \cos\theta>0$ the profile curve $\gamma(s)$ can be written in the form $(x, u(x))$, by (\ref{0316eq2.5}), the function $u(x)$ satisfies the differential equation
	\begin{equation}\label{0528eq2.7}
	\frac{u''}{1 + (u')^2} = x\,u' - u + \frac{n-1}{u} + \lambda\sqrt{1+(u')^2}.
	\end{equation}
	It has a unique constant solution $u = \frac{\lambda + \sqrt{\lambda^2 + 4(n - 1) }}{2}$ and a unique ``round solution'' with radius $\frac{\lambda + \sqrt{\lambda^2 + 4n }}{2}$, namely, \[ u(x) = \left( \left( \frac{\lambda + \sqrt{\lambda^2 + 4n }}{2}\right)^2 - x^2 \right)^{\frac{1}{2}} .\]
	We henceforth denote $\frac{\lambda + \sqrt{\lambda^2 + 4(n - 1) }}{2}$ and $\frac{\lambda + \sqrt{\lambda^2 + 4n }}{2}$ by $C_\lambda$ and $R_\lambda$ respectively.

	When $\dot{x} = \cos\theta<0$ the profile curve $\gamma(s)$ can also be written in the form $(x, u(x))$, by (\ref{0316eq2.5}), the function $u(x)$ satisfies the differential equation
	\begin{equation}\label{240111eq2.8}
	\frac{u''}{1 + (u')^2} = x\,u' - u + \frac{n-1}{u} - \lambda\sqrt{1+(u')^2}.
	\end{equation}
	It has a unique constant solution $u =C_{-\lambda} = \frac{-\lambda + \sqrt{\lambda^2 + 4(n - 1) }}{2}$ and a unique ``round solution'' with radius $R_{-\lambda} = \frac{-\lambda + \sqrt{\lambda^2 + 4n }}{2}$, namely, \[ u(x) = \sqrt{R_{-\lambda}^2 -x^2}.\]

	When $\dot{r} = \sin\theta>0$ the profile curve can be written in the form $(f(r), r)$, by (\ref{0316eq2.5}), the function $f(r)$ satisfies the differential equation
	\begin{equation}\label{0526eq2.8}
	\frac{f''}{1 + (f')^2} =  ( r - \frac{n-1}{r}  )f' -f - \lambda\sqrt{1+(f')^2}.
	\end{equation}
	It has a unique constant solution $f = -\lambda$.

	When $\dot{r} = \sin\theta<0$ the profile curve can also be written in the form $(f(r), r)$, by (\ref{0316eq2.5}), the function $f(r)$ satisfies the differential equation
	\begin{equation}\label{231212eq2.9}
	\frac{f''}{1 + (f')^2} =  ( r - \frac{n-1}{r}  )f' -f + \lambda\sqrt{1+(f')^2}.
	\end{equation}
	It has a unique constant solution $f = \lambda$.
	Differentiating  the above two equations with respect to $r$, we obtain
	\begin{equation}\label{0526eq2.9}
	\frac{f'''}{1 + (f')^2} = \frac{2f'(f'')^2}{(1 + (f')^2)^2}+ ( r - \frac{n-1}{r}  )f'' + \frac{n-1}{r^2}f' -\frac{\lambda f'f''}{\sqrt{1 + (f')^2}}.
	\end{equation}
	\begin{equation}\label{231212eq2.10}
	\frac{f'''}{1 + (f')^2} = \frac{2f'(f'')^2}{(1 + (f')^2)^2}+ ( r - \frac{n-1}{r}  )f'' + \frac{n-1}{r^2}f' +\frac{\lambda f'f''}{\sqrt{1 + (f')^2}}.
	\end{equation}
	We conclude this section with two lemmas about the solutions to (\ref{0526eq2.8}) and (\ref{231212eq2.9}).
	\begin{lem}\label{0527lem2.1}
		Let  $f: (a,b) \to \mathbb{R}$ be a solution to \eqref{0526eq2.8} or \eqref{231212eq2.9}, $c \in (a,b)$ with $f''(c) = 0$. If $f'(c) = 0$ or $f'''(c) = 0$, then $f$ is constant. Moreover, $f'(c)f'''(c)>0$ provided $f$ is not constant.
	\end{lem}
	\begin{proof}
		Suppose $f$ is a solution to \eqref{0526eq2.8}. If $f'(c) = f''(c) = 0$, it follows from \eqref{0526eq2.8} that $f(c) = -\lambda$. Thus, by the uniqueness of solutions to second-order ordinary differential equations, we conclude that $f = -\lambda$.
		If $f''(c) = f'''(c) = 0$, equation \eqref{0526eq2.9} implies that $f'(c) = 0$. Consequently, the preceding argument yields $f = -\lambda$.
		Finally, assume $f''(c) = 0$ and $f$ is not constant. By what we have just shown, we must have $f'(c) \neq 0$. Evaluating equation \eqref{0526eq2.9} at $c$ and multiplying both sides by $f'(c)$ immediately yields $f'(c)f'''(c) > 0$.
		The proof for the case where $f$ is a solution to \eqref{231212eq2.9} is entirely analogous.
	\end{proof}

	\begin{lem}\label{lem2.2}
		Given an integer \( n \geq 2 \) and \( \lambda, b \in \mathbb{R} \), then there exist \( \delta > 0 \) and a unique analytic function \( f(r; \varepsilon) \) defined on \([0, \delta] \times [-\delta, \delta]\) such that for each \( \varepsilon \in [-\delta, \delta] \), the function \( f(\cdot; \varepsilon) \) is a solution of the following singular initial value problem:
\[
\begin{cases}
\dfrac{f''}{1 + (f')^2} = \left(r - \dfrac{n-1}{r}\right)f' - f \pm \lambda \sqrt{1 + (f')^2} \\[2ex]
f(0) = b + \varepsilon, \quad f'(0) = 0.
\end{cases}
\]
Moreover, if \( \tilde{f}(\cdot; \varepsilon) \in C^1([0, \delta]) \cap C^2((0, \delta]) \) is also a solution of the above initial value problem, then \( \tilde{f}(\cdot; \varepsilon) = f(\cdot; \varepsilon) \).
	\end{lem}

\begin{proof}
	Let \( y(r) = f(r) - b - \varepsilon \), \( p(r) = y'(r) \). Then \( y \) and \( p \) satisfy
\[
y' = p
\]
\[
p' = \left[ (r^2 - n + 1)p - ry - rb - r\varepsilon \pm \lambda r \sqrt{1 + p^2} \right](1 + p^2) := \phi(r, y, p, \varepsilon).
\]
Clearly, \( \phi \) is analytic in the variables \( r, y, p, \varepsilon \), and we have
\[
\frac{\partial \phi}{\partial p}(0, 0, 0, 0) = -n + 1 \notin \mathbb{N}, \quad \frac{\partial \phi}{\partial y}(0, 0, 0, 0) = 0,
\]
\[
\frac{\partial \phi}{\partial \varepsilon} = -r(1 + p^2) \implies \frac{\partial^k \phi}{\partial \varepsilon^k}(0, 0, 0, 0) = 0, \text{ for any positive integer } k.
\]
Now applying proposition 1 in \cite{Hs5}, we obtain the existence, uniqueness, and dependence of the analytic solution.

Furthermore, since \( f \) is analytic,  \( f \) clearly belongs to the function space \( C^1([0, \delta]) \cap C^2((0, \delta]) \). Applying theorem 2.2 in \cite{Liang} yields \( \tilde{f} = f \).
\end{proof}

\section{Behavior of the solutions to (\ref{0526eq2.8}) and (\ref{231212eq2.9})}
To construct the desired compact $\lambda$-hypersurface in theorem \ref{0316thm1.1}, our strategy relies on a shooting method. Before applying this method, we must first rigorously classify the local geometric shapes of the profile curves. Therefore,
in this section, we analyze the qualitative behavior of the profile curves governed by equations (\ref{0526eq2.8}) and (\ref{231212eq2.9}). To geometrically classify the shapes of these curves, it is essential to understand the signs of the derivatives $f'$ and $f''$. The following lemma provides a fundamental property of the product $f'f''$.
	\begin{lem}\label{0527lem3.1}
		Let  $f: (a,b) \to \mathbb{R}$ be a solution to \eqref{0526eq2.8} or \eqref{231212eq2.9}, $c \in (a,b)$, we have the following assertions:
		\begin{enumerate}[1.]
			\item 	if $f'(c) f''(c) <0$, then $f'(r)f''(r) <0$ for $r \in (a, c]$,
			\item   if $f'(c)f''(c) > 0$, then $f'(r)f''(r) >0$ for $r \in [c, b)$.
		\end{enumerate}
	\end{lem}
	\begin{proof}
		We adopt a contradiction argument from \cite{D}.
		We only prove the first part of the lemma; the second part can be proved similarly.
		If $f'(c) > 0,\ f''(c) <0$, we want to show that $f''(r) <0$ for  $r \in (a, c]$ which yields that $f'(r) > 0$ and then $f'(r)f''(r) <0$ for  $r \in (a, c]$.
		By the continuity of $f$, we know that $f''(r) <0$ when $r$ is near $c$.
		Hence, if $f''(r) \ge 0$ for some $r \in (a, c]$, then there must be a point $\bar{c} \in (a,c)$ such that $f''(\bar{c}) = 0$ and $f''(r) < 0$ for $r \in (\bar{c}, c]$.
		The choice of $\bar{c}$ implies that $f'(\bar{c}) > f'(c) > 0$ and $f'''(\bar{c}) \le 0$ which contradicts lemma \ref{0527lem2.1}.
		If $f'(c) < 0,\ f''(c) >0$, the proof is similar to the above.
	\end{proof}

  This lemma, together with lemma \ref{0527lem2.1}, shows the following corollary.
	\begin{cor}\label{0528cor3.1}
		Let $f(a, b) \to \mathbb{R}$ be a solution to \eqref{0526eq2.8} or \eqref{231212eq2.9} which is not constant, let $c \in (a,b)$.
		\begin{enumerate}[1.]
			\item 	If $f'(c) =0$, then $f''(r) < 0$ or $f''(r) > 0$  for $r$ in $(a,b)$.
			In the former case, we have that $f'(r)>0$ for $r \in (a,c)$ and that $f'(r) < 0$ for $r \in (c,b)$.
			In the latter case, we have that $f'(r)<0$ for $r \in (a,c)$ and that $f'(r) > 0$ for $r \in (c,b)$.
			\item   If $f''(c) =0$, then $f'(r) < 0$ or $f'(r) > 0$  for $r$ in $(a,b)$.
			In the former case, we have that $f''(r) > 0$ for $r \in (a,c)$ and that $f''(r) < 0$ for $r \in (c,b)$.
			In the latter case, we have that $f''(r) < 0$ for $r \in (a,c)$ and that $f''(r) > 0$ for $r \in (c,b)$.
		\end{enumerate}
	\end{cor}

Corollary \ref{0528cor3.1} ensures that the critical points and inflection points of the profile curve are strictly separated. This observation allows us to rigorously categorize the solutions into three distinct geometric classes as follows.
\begin{defn}\label{def3.1}
	Let $f: (a, b) \to \mathbb{R}$ be a solution to \eqref{0526eq2.8} or \eqref{231212eq2.9} which is not constant.
	\begin{enumerate}[1.]
	\item $f$ is said to be of type 1 if there is a unique point in $(a,b)$ such that $f'=0$.
	\item $f$ is said to be of type 2 if there is a unique point in $(a,b)$ such that $f''=0$.
	\item $f$ is said to be of type 3 if there is no point in $(a,b)$ such that $f'f''=0$.
	\end{enumerate}
\end{defn}

Having classified the solutions into distinct classes, we now investigate their boundary behaviors. The next few propositions (proposition \ref{231206lem3.5}, \ref{prop:24-7-14}, \ref{231204lem3.3}, \ref{240112lem3.3}, \ref{231211lem3.3}) determine the maximal existence intervals and the limits of the solutions at their endpoints depending on the signs of $f'$ and $f''$.

If $f'(r)f''(r)$ is positive for $r$ close to the left endpoint of the maximal existence intervals, we have:
\begin{prop}\label{231206lem3.5}
	Let $f: (a, b) \to \mathbb{R}$ be a left maximally extended solution to \eqref{0526eq2.8} or \eqref{231212eq2.9}, if $f'f'' > 0$ on $(a, b)$, then $a=0$ and $\lim_{r \to 0^+}f'(r)=0$.
\end{prop}
\begin{proof}
	If $f' < 0$ and $f''< 0$ on $(a, b)$, then $\lim_{r \to a^+}f(r)$ and $\lim_{r \to a^+}f'(r)$ are both finite, which follows that $a=0$ due to that $f: (a, b) \to \mathbb{R}$ is maximally extended.
	Now, by (\ref{0526eq2.8}) or (\ref{231212eq2.9}), $\lim_{r \to 0^+}f'(r)$ must be zero since $\lim_{r \to 0^+}f'(r)<0$ leads to $\lim_{r \to 0^+}f''(r)=\infty.$
	If $f' > 0$ and $f''> 0$ then the argument is similar to that above.
\end{proof}

If $f'(r)f''(r)$ is positive for $r$ close to the right endpoint of the maximal existence intervals, we have:
\begin{prop}\label{prop:24-7-14}
	Let $f: (a, b) \to \mathbb{R}$ be a right maximally extended solution to \eqref{0526eq2.8}. If $f' > 0$ and $f''> 0$ on $(a, b)$, then $b > C_\lambda$; if $f' < 0$ and $f''< 0$ on $(a, b)$, then $b > C_{-\lambda}$.
\end{prop}
\begin{proof}
	We adopt a contradiction argument  from \cite{D}. We only consider the case of $\lambda <0$, the proof for the case of $\lambda \ge 0$ is similar. Suppose $f' > 0$, $f''> 0$ on $(a, b)$ and $b \le C_\lambda$, then $\lim_{r \rightarrow b}f(r) = \infty$ or $\lim_{r \rightarrow b}f'(r) = \infty$ (otherwise $f$ is not right maximally extended), that is $f$ blows up at $b$. If $b < C_\lambda$, by (\ref{0526eq2.8}) we have
	\begin{align}
	\frac{f''}{1 + (f')^2}
	&=  ( r - \frac{n-1}{r}  )f' -f - \lambda\sqrt{1+(f')^2}\nonumber\\
	&<   ( r - \frac{n-1}{r}  )f' -f - \lambda (1+f') \nonumber\\
	&=   ( r - \frac{n-1}{r} -\lambda )f' -f - \lambda,\nonumber
	\end{align}
	which implies that $f'' < 0$ as $r \rightarrow b$.
	This contradicts $f''>0$ on $(a,b)$.
	On the other hand, the existence of the constant solution to (\ref{0528eq2.7}) prevents $b$ from being equal to $C_\lambda$.
	To see this, suppose $b = C_\lambda$. The previous calculation implies that
	\[
	f
	< ( r - \frac{n-1}{r} - \lambda )f' - \frac{f''}{1 + (f')^2} - \lambda\\
	<   - \lambda
	\] for $r \in (a,b)$, and therefore there exists $x_* \le -\lambda $ such that $\lim_{r \rightarrow b}f(r) = x_*$.
	Near the point $(x_*, b)$, we write the curve $(f(r), r)$ as $(x,u(x))$, where $u$ satisfies the differential equation (\ref{0528eq2.7}).
	Now, $u(x_*) = C_\lambda$ and $u'(x_*) = 0$, and by the uniqueness of solutions to equation (\ref{0528eq2.7}), $u$ must be the constant function $u(x) = C_\lambda$. This contradicts the fact  that $(f(r), r)$ agrees with $(x,u(x))$ near $(x_*,b)$. The second part of the lemma can be proved similarly.
\end{proof}

If $f'(r)f''(r)$ is negative for $r$ close to the left endpoint of the maximal existence intervals, we have:
\begin{prop}\label{231204lem3.3}
	Let $\lambda<0$ and $f: (a, b) \to \mathbb{R}$ be a left maximally extended solution to \eqref{231212eq2.9}.
	If $f' < 0$ and $f''> 0$ on $(a, b)$, then $0<a < C_\lambda$ and $\lim_{r \to a^+}f(r)<\infty$; if $f' > 0$ and $f''< 0$ on $(a, b)$, then $0<a < C_{-\lambda}$ and $\lim_{r \to a^+}f(r)>-\infty$.
\end{prop}
\begin{proof}
	Assume $f' < 0$ and $f''> 0$ on $(a, b)$. The proof of $a < C_\lambda$ is similar to that of proposition \ref{prop:24-7-14}.

	Next we prove that $\lim_{r \to a^+}f(r)<\infty$ using a method from \cite{D}.
	Choose $c_1$ in $(a,\min\left\lbrace b, C_\lambda\right\rbrace )$.
	Then, for $r\le c_1$, we have $r-\dfrac{n-1}{r}-\lambda \le -\dfrac{1}{M_1}$ for some $M_1>0$.
	Let $\phi = - f' $. Then $\phi>0$ and $\phi'<0$ for $r \in (a,b)$.
	Using (\ref{231212eq2.10}), we have\[ \phi'' \ge -\dfrac{1}{M_1}\phi'\phi^2\] when $r\in (a,c_1)$.
	Fix $\varepsilon>0$, and let\[ \phi_\varepsilon(r) = \dfrac{\phi(c_1)\sqrt{c_1-(a+\varepsilon)}+\sqrt{3M_1}}{\sqrt{r-(a+\varepsilon)}}.\]
	Then \[ \phi_\varepsilon'' = -\dfrac{3}{2}\dfrac{1}{(\phi(c_1)\sqrt{c_1-(a+\varepsilon)}+\sqrt{3M_1})^2}\phi_\varepsilon'\phi_\varepsilon^2 \le -\dfrac{1}{M_1}\phi_\varepsilon'\phi_\varepsilon^2\] for $r\in (a+\varepsilon, c_1)$.
	Note that $\phi_\varepsilon(c_1)> \phi(c_1)$ and $\phi_\varepsilon(a + \varepsilon)> \phi(a + \varepsilon)$.
	If $\phi> \phi_\varepsilon$ at some point in $(a+\varepsilon, c_1)$, then $\phi - \phi_\varepsilon$ must achieve a positive maximum in $(a+\varepsilon, c_1)$.
	At such a point, \[ 0 \ge (\phi - \phi_\varepsilon)'' \ge -\dfrac{1}{M_1}\phi'(\phi^2 - \phi_\varepsilon^2) > 0, \] which is a contradiction.
	Therefore, $\phi_\varepsilon \ge \phi$.
	Taking $\varepsilon \to 0$, we have the estimate \[ \phi(r) \le \dfrac{\phi(c_1)\sqrt{c_1-a}+\sqrt{3M_1}}{\sqrt{r-a}} \] for $r \in (a, c_1)$.
	Integrating from $r$ to $c_1$,  we have \[ f(r) - f(c_1) \le 2(\phi(c_1)\sqrt{c_1-a}+\sqrt{3M_1})(\sqrt{c_1 - a} -\sqrt{r -a}). \] and then \[ \lim_{r \to a^+}f(r) \le 2(\phi(c_1)\sqrt{c_1-a}+\sqrt{3M_1})(\sqrt{c_1 - a}) + f(c_1). \]

	Now we can show $a>0$ as in \cite{D}.
	Let $\beta(r) = -\arctan f'(r)$.
	Then \[
	\begin{aligned}
	\dfrac{\rm d}{{\rm d}r}(\log\sin \beta(r)) &= r-\dfrac{n-1}{r}-\dfrac{f(r)}{f'(r)}+\dfrac{\lambda\sqrt{1+f'(r)^2}}{f'(r)}\\
	&\le r-\dfrac{n-1}{r}-\dfrac{f(r)}{f'(r)} + \dfrac{\lambda(1-f'(r))}{f'(r)}\\
	&\le r-\dfrac{n-1}{r}-\lambda - \dfrac{A-\lambda}{f'(c_1)}
	\end{aligned}\]
	for $r \in (a,c_1)$, where $A=\max\left\lbrace \lambda, 2(\phi(c_1)\sqrt{c_1-a}+\sqrt{3M_1})(\sqrt{c_1 - a}) + f(c_1) \right\rbrace$.
	Integrating the inequality from $r$ to $c_1$: \[ \log\left( \dfrac{\sin\beta(c_1)}{\sin\beta(r)}\right)  \le \dfrac{1}{2}c_1^2+(n-1)\log\left( \dfrac{r}{c_1}\right)-\left( \lambda+\dfrac{A-\lambda}{f'(c_1)} \right)c_1.  \]
	Thus\[ \left(\dfrac{r}{c_1} \right)^{n-1} \ge \sin\beta(c_1)\,\exp\left\lbrace -\dfrac{1}{2}c_1^2 + \left( \lambda+\dfrac{A-\lambda}{f'(c_1)} \right)c_1  \right\rbrace >0  \] for $r$ in $(a,c_1)$.
	In particular, $a>0$.
	This completes the proof of the first part of the proposition.

	Assume $f' > 0$ and $f''< 0$ on $(a, b)$.
	Then the proof of $a < C_{-\lambda}$ is also similar to that of proposition \ref{prop:24-7-14}.

	We first show that $\lim_{r \to a^+}f'(r) = \infty$, which we will consider in two cases.
	If $\lim_{r \to a^+}f(r)$ is negative infinity, then $\lim_{r \to a^+}f'(r) = \infty$ follows from Lagrange Mean Value Theorem.
	If $\lim_{r \to a^+}f(r)$ is finite, then $a>0$ as in the first part of this proposition.
	Therefore, $\lim_{r \to a^+}f'(r) = \infty$ since $f: (a, b) \to \mathbb{R}$ is left maximally extended.

	Choose $\tilde{\lambda}\in(\lambda, 0)$ such that $C_{-\tilde{\lambda}}= \dfrac{\max\left\lbrace a,\sqrt{n-1}\right\rbrace  +C_{-\lambda} }{2}$.
	Then \[ r-\dfrac{n-1}{r} + \tilde{\lambda} < 0 \] for $0<r<C_{-\tilde{\lambda}}$.
	Since $\lim_{r \to a^+}f'(r) = \infty$, there is a sufficiently small $\varepsilon_1>0$ such that \[ \dfrac{f'}{\sqrt{1+(f')^2}}> \dfrac{\tilde{\lambda}}{\lambda}\] when $r \in (a,a+\varepsilon_1)$.
	Now, choose $c_2 \in (a, \min\left\lbrace b, C_{-\tilde{\lambda}},\, a+\varepsilon_1  \right\rbrace )$.
	Then, for $r\le c_2$, we have $r-\dfrac{n-1}{r}+\tilde{\lambda} \le -\dfrac{1}{M_2}$ for some $M_2>0$.
	Let $\varphi =  f' $.
	Then $\varphi>0$ and $\varphi'<0$ for $r \in (a,b)$.
	The rest of the argument is similar to that of the first part of this proposition.
\end{proof}

Together with proposition \ref{231206lem3.5}, this proposition gives
\begin{cor}\label{240113cor3.3}
	Let $\lambda<0$ and $f: (a, b) \to \mathbb{R}$ be a maximally extended solution to \eqref{231212eq2.9}.  Then $a < C_{-\lambda}$.
\end{cor}

If $f'(r)f''(r)$ is negative for $r$ close to the right endpoint of the maximal existence intervals, we have the following proposition.
The proof is trivial so we omit it.
\begin{prop}\label{240112lem3.3}
	Let $f: (a, b) \to \mathbb{R}$ be a right maximally extended solution to \eqref{0526eq2.8} or \eqref{231212eq2.9}. If $f'f'' <0$, then $b = \infty$.
\end{prop}

The following proposition can be interpreted as a complement to proposition \ref{prop:24-7-14}.
	\begin{prop}\label{231211lem3.3}
		Let $\lambda<0$ and $f: [a, b) \to \mathbb{R}$ be a right maximally extended solution to \eqref{0526eq2.8}. If $f' > 0$ and $f''> 0$ on $[a, b)$, then $C_{\lambda}<b < \infty$ and $\lim_{r \to b^-}f(r)<\infty$.
	\end{prop}

The proof of proposition \ref{231211lem3.3} is quite involved and requires careful asymptotic estimates to rule out the case where the solution extends to infinity. To this end, we first prepare a lemma that establishes an upper bound for the maximal existence interval of a related initial value problem.
\begin{lem}\label{lem1}
	If $\lambda<0$, $x_0 \in (-\lambda,\infty)$ and $(x_0, x_{\infty})$ is a maximally extended solution to the initial value problem \begin{equation}\label{eq3.1}
	\left\lbrace \begin{aligned}
	& u'' = \left[ xu'-u+\dfrac{n-1}{u}+\lambda\sqrt{1+(u')^2}\right] (1+(u')^2),\\
	& u(x_0) = u_0,\\
	& u'(x_0) \ge \sigma,
	\end{aligned}\right.
	\end{equation}
	where $u_0>0$, $\sigma>0$ and they satisfy $u_0 = \sigma x_0 + \lambda\sqrt{1+\sigma^2}$. Then $x_{\infty}<\infty$. Furthermore, if $u_0 \ge \sqrt{n-1}$, then $x_\infty < -\lambda + \dfrac{n}{n-1}(x_0+\lambda)$.
\end{lem}
\begin{proof}
	Defining $\Psi(x) := xu'(x)  + \lambda\sqrt{1+(u')^2} - u(x)$, we note that the initial conditions are equivalent to $u(x_0)>0$ and $\Psi(x_0)\ge 0$ since $t\,x_0 + \lambda\sqrt{1+t^2}$ is strictly increasing with respect to $t$.
	Also, $\Psi(x_0) \ge 0$ implies that $u'(x_0)>0$, $u''(x_0)>0$, and $x_0 + \dfrac{\lambda u'(x_0)}{\sqrt{1+u'(x_0)^2}}>x_0+\lambda >0$.
	Since \begin{equation}
	\Psi' = \left( x + \dfrac{\lambda u'}{\sqrt{1+(u')^2}} \right)\left( \Psi + \dfrac{n-1}{u}  \right)(1+(u')^2),
	\end{equation}
	we see that $\Psi \ge 0$ and $\Psi' > 0$ in $[x_0, x_{\infty})$.
	Thus $u'>0$ and $u''>0$ in $[x_0, x_{\infty})$. Consequently, without loss of generality, we may assume $u(x_0) \ge \sqrt{n-1}$.

	Define the quantity \[ \Phi = \Psi +\dfrac{n-1}{u}. \]
	Then we have $\Phi(x_0)\ge \dfrac{n-1}{u_0}>\dfrac{n-1}{\sigma(x_0+\lambda)}$ since $u_0 = \sigma x_0 + \lambda\sqrt{1+\sigma^2}$.
	We claim that in fact $\Phi(x)>\dfrac{n-1}{\sigma(x_0+\lambda)}$ for all $x\ge x_0$.
	This is because, assuming this holds up to $x$, we have for $x\ge x_0$, \[ \begin{aligned}
	\Phi' &= \left( x + \dfrac{\lambda u'}{\sqrt{1+(u')^2}} \right)\Phi\,(1+(u')^2) - \dfrac{n-1}{u^2}u'\\
	&> (x_0 + \lambda)\dfrac{n-1}{\sigma(x_0+\lambda)}(1+\sigma u') - u'\\
	&\ge \dfrac{n-1}{\sigma} >0.
	\end{aligned}  \]
	We thus obtain the inequality $u'' >\dfrac{n-1}{\sigma(x_0 + \lambda)}(1+(u')^2)$ for $x\ge x_0$.
	Integrating this inequality gives \[ u'(x) \ge \tan \left[ (n-1)\dfrac{x-x_0}{\sigma(x_0+\lambda)} + \arctan \sigma\right] , \]
	which finally leads to $x_\infty \le \dfrac{\sigma(x_0+\lambda)}{n-1}\left( \dfrac{\pi}{2} - \arctan \sigma\right) +x_0 < \dfrac{x_0+\lambda}{n-1} +x_0 = -\lambda + \dfrac{n}{n-1}(x_0+\lambda)$.
\end{proof}
To further control the growth of the solution $u(x)$ as $x \to \infty$, we now derive a crucial integral identity by viewing the nonlinear ODE as a linear inhomogeneous equation.
\begin{lem}[Integral identity]\label{lem3.2}
	Let $\lambda<0$, $d> -\lambda$.
	Let $u:[d, \infty) \to \R^+$  be a solution to  \eqref{0528eq2.7} such that $u'> 0$ and $u'' \ne 0$. Then, for some $\sigma = \sigma(u) \ge 0$, \(u\) satisfies the identity \begin{equation}\label{eq3.3}
	\begin{aligned}
	u(x) = \sigma x
	+x\int_{x}^{\infty}\dfrac{1}{t^2}\left\lbrace \int_{t}^{\infty}sA(s)B(s)e^{  -\int_{t}^{s}zB(z){\rm d}z} {\rm d}s\right\rbrace {\rm d}t,
	\end{aligned}
	\end{equation} when $x \in [d,\infty)$.
	Here $A(s) = \dfrac{n-1}{u(s)} + \lambda\sqrt{1+(u'(s))^2}$, $B(s) =1+(u'(s))^2$.
	In particular, there is an $M>0$ such that $|u(x) -\sigma x|<M$ for $x \in [d,\infty)$.
\end{lem}
\begin{proof}
	Suppose first that we are given a solution $u: [d,a) \to \mathbb{R}^+$ over an interval $[d,a)$.
	We can regard $u$ as a solution to an inhomogeneous linear equation determined by freezing the coefficients at $u$,
	\begin{equation}
	u'' -x(1+(\varphi')^2)u' + (1+(\varphi')^2)u = \left( \dfrac{n-1}{\varphi} + \lambda\sqrt{1+(\varphi')^2}\right) (1+(\varphi')^2),
	\end{equation}
	where we have set $u = \varphi$.
	We can solve the resulting linear equation with variable coefficients, for $x \in [d,a)$, by observing that a fundamental pair of solutions to the homogeneous linear equation is given by
	\begin{equation}
	u_1(x) = x\ \ \ \ {\rm and}\ \ \ \ u_2(x)= x\int_{x}^{a}\frac{\exp\left( -\int_{s}^{a}z(1+(\varphi')^2){\rm d}z\right)}{s^2}{\rm d }s.
	\end{equation}
	Then computing the Wronskian $W(s) = \exp\left( -\int_{s}^{a}z(1+(\varphi')^2){\rm d}z\right)$ and matching the initial conditions gives
	\begin{equation}\label{eq3.6}
	\begin{aligned}
	u(x) &= \dfrac{u(a)}{a}x +(u(a) - u'(a)a)x\int_{x}^{a}\frac{\exp\left( -\int_{s}^{a}z(1+(\varphi')^2){\rm d}z\right)}{s^2}{\rm d }s\\
	+x\int_{x}^{a}\dfrac{1}{t^2}&\left\lbrace \int_{t}^{a}s\left( \dfrac{n-1}{u(s)} + \lambda\sqrt{1+(u'(s))^2}\right) (1+(u'(s))^2) e^{  -\int_{t}^{s}z(1+(u'(z))^2){\rm d}z} {\rm d}s\right\rbrace {\rm d}t
	\end{aligned}
	\end{equation}
	To complete the proof of (\ref{eq3.3}), we will show that for some limit $\sigma \ge 0$,
	\begin{equation}
	\dfrac{u(a)}{a} \to \sigma, \ \ {\rm for}\ \ \ a \to \infty,
	\end{equation}
	\begin{equation}
	(u(a) - u'(a)a)x\int_{x}^{a}\frac{\exp\left( -\int_{s}^{a}z(1+(\varphi')^2){\rm d}z\right)}{s^2}{\rm d }s \to 0, \ \ {\rm for}\ \ \ a \to \infty.
	\end{equation}
	Recall that by lemma \ref{lem1}, for any solution $u: [d,\infty) \to \R^+$ the quantity $\Psi(x) = xu'(x) - u(x) + \lambda\sqrt{1+(u')^2} $ is pointwise negative.
	Using $u' > 0$ and $\sqrt{1+(u')^2} < 1 + u'$, we have that  the ratio $\frac{u(a)-\lambda}{a+\lambda}$ is monotonically decreasing in $a$, and hence converges to some limit $\sigma \ge 0$.
	Then $\lim_{a \to \infty}\frac{u(a)}{a} = \sigma$, which implies $\lim_{a \to \infty}u'(a) = \sigma$ by the fact that $u'' \ne 0$ and the Lagrange Mean Value Theorem.
	Also, there is a sequence $\left\lbrace a_k\right\rbrace $ increasing to $\infty$ such that $\lim_{k \to \infty}u''(a_k) =0$.
	Thus $u(a_k) - u'(a_k)a_k$ converges to a finite limit by (\ref{0528eq2.7}) as $k$ approaches infinity, so that this term is bounded, and since
	\[ \int_{x}^{a_k}\frac{\exp\left( -\int_{s}^{a_k}z(1+(\varphi')^2){\rm d}z\right)}{s^2}{\rm d }s \le \dfrac{e^{-\frac{1}{2}a_k^2}}{x^2}\int_{x}^{a_k}e^{\frac{1}{2}s^2}{\rm d}s \to 0 \ \ {\rm as} \ \,k \to \infty,\]
	we see that inserting the sequence $\left\lbrace a_k\right\rbrace \to \infty$  in (\ref{eq3.6}) leads to (\ref{eq3.3}).

	Since $u'>0$ and $\lim_{x \to \infty}u'(x)$ is finite, there is an $M>0$ such that $|A(x)|<M$ for all $x \in [d, \infty)$.
	Using (\ref{eq3.3}), we get \[
	|u(x) -\sigma x| \le Mx\int_{x}^{\infty}\dfrac{1}{t^2}\left\lbrace \int_{t}^{\infty}sB(s)e^{  -\int_{t}^{s}zB(z){\rm d}z} {\rm d}s\right\rbrace {\rm d}t  = M.
	 \]\end{proof}

 \noindent{\it Proof of proposition \ref{231211lem3.3}.}
    It has been proved in proposition \ref{prop:24-7-14} that $b>C_{\lambda}$.
	We first prove $b<\infty$.
	Suppose for the sake of contradiction that $b = \infty$.
	Then $\lim_{r \to \infty}f(r) = \infty$ by the convexity of $f$.
	On one hand, the graph of $f$ can be written as a graph of a function over the $x$-axis, denoted by $u(x)$ for $x \in [f(a), \infty)$.
	Then $u$ satisfies (\ref{0528eq2.7}), $u'>0$, $u''<0$  and $\lim_{x \to \infty}u(x) = \infty.$
	Using lemma \ref{lem3.2}, there are $\sigma>0$, $M>0$ such that $|u(x)-\sigma x|<M$.
	Hence we must have $f(r)< \frac{1}{\sigma}(r+M)$ for $r$ in $[a, \infty)$.
	On the other hand, through (\ref{0526eq2.9}), there exists sufficiently large $c_1>0$ such that $f'''(r)>f''(r)$ for $r \in [c_1, \infty)$.
	Integrating from $c_1$ to $r$, we get $f(r) \ge Ae^r +B\,r +C$, where $A>0$, $B$, $C$ are constants.
	This contradicts the fact that $f(r)< \frac{1}{\sigma}(r+M)$. Hence $b<\infty$.

	Using the Lagrange Mean Value Theorem and the condition that $f$ is right maximally extended, we know that $\lim_{r\to b^-}f'(r) = \infty$.
	Choose $\tilde{\lambda}$: $\lambda< \tilde{\lambda}<0$ such that $C_{\tilde{\lambda}}= \frac{C_\lambda + \min\left\lbrace \sqrt{n-1},b \right\rbrace }{2}$, then \[ r-\dfrac{n-1}{r} - \tilde{\lambda} > 0 \] for $r>C_{\tilde{\lambda}}$.
	Since $\lim_{r \to b^-}f'(r) = \infty$, there is a sufficiently small $\varepsilon_1>0$ such that \[ \dfrac{f'}{\sqrt{1+(f')^2}}> \dfrac{\tilde{\lambda}}{\lambda}\] when $r \in (b-\varepsilon_1,b)$.
	Now, choose $c_2 \in (\max\left\lbrace a, C_{\tilde{\lambda}},\, b-\varepsilon_1  \right\rbrace, b )$. Then, for $r\ge c_2$, we have $r-\dfrac{n-1}{r}-\tilde{\lambda} \ge \dfrac{1}{M}$ for some $M>0$.
	Let $\phi =  f' $. Then $\phi>0$ and $\phi'>0$ for $r \in (a,b)$.
	Using (\ref{0526eq2.9}), we have \[ \phi''\ge \dfrac{1}{M}\phi'\phi^2\] for $r \in [c_2,b)$.
	The rest of the argument is similar to that of proposition \ref{231204lem3.3}. $\hfill\square$

Before concluding this section, we establish a specific technical lemma regarding the existence of a critical point ($f'=0$) near the endpoint. While isolated here, this property is a crucial geometric tool. It will be utilized in section 4 (lemma \ref{231207lem4.2}) to determine the shape of the profile curves, and more importantly, it will serve as a key ingredient in the final proof of theorem \ref{0902thm1.1} to rule out certain degenerate configurations (specifically, to show that the curve $\gamma_{\delta_s}$ does not belong to \(\mathcal{C}_1(3)\)).
\begin{lem}\label{231205lem3.4}
	Let $\lambda<0$ and $f: [b-\sqrt{2}, b) \to \mathbb{R}$ be a solution to \eqref{231212eq2.9} with $\sqrt{n-1}+\sqrt{2}\le b<\infty$. If $f' < 0$ in a neighborhood of $b$ and $f(b)>0$, then there exists a point $c$ in $[b-\sqrt{2}, b)$ such that $f'(c) = 0$.
\end{lem}
\begin{proof}
	We adopt a contradiction argument from \cite{D}.
	Suppose $f'(r)<0$ for $r$ in $[b-\sqrt{2},b)$.
	Then $f>0$ for $r$ in $[b-\sqrt{2},b)$ since $f(b)>0$, and therefore $f''<0$ in $[b-\sqrt{2},b)$, $f''(b-\sqrt{2}) \le -f(b-\sqrt{2})$ by (\ref{231212eq2.9}).
	Using (\ref{231212eq2.10}), we also have $f'''<0$ in $[b-\sqrt{2},b)$.
	Integrating from $b-\sqrt{2}$ to $r$, we get \[ f(r) \le f(b-\sqrt{2})\left[ 1- \frac{1}{2}(r-(b-\sqrt{2}))^2 \right] \] for $r$ in $[b-\sqrt{2},b)$.
	This tells us that $f(b) \le 0$, which is a contradiction.
\end{proof}

\section{Global splicing and classification of profile curves}

	In the previous discussion, we analyzed the local behavior of the profile curves where they can be written as graphs over the $r$-axis. However, a profile curve may bend and lose its graph representation, which occurs when $\dot r=\sin\theta = 0$. To study the global geometry of these curves, we need to track these turning points and piece the local graphical solutions together.
	Now we go back to the profile curves  $\gamma_\delta(s) = P(\Gamma_\delta(s))$.
	For $0<\delta < C_\lambda$, we note that if $\sin \theta_\delta(c) = 0$ then $\dot{\theta}_\delta(c) \ne 0$ otherwise $r_\delta(s)=C_\lambda$ or $r_\delta(s)=C_{-\lambda}$ for all $s$ by the uniqueness theorem for the system (\ref{0316eq2.5}).
	As a consequence, $\sin \theta_\delta(s) \ne 0$ in a deleted neighborhood of $c$, and the sign of $\sin \theta_\delta$ in $(c-\varepsilon,c)$ is different from that of $\sin \theta_\delta$ in $(c,c+\varepsilon)$.
    We are able to give the following definition as in \cite{A,CW1, LCW}.
	\begin{defn}\label{231116def3.1}
		For $0<\delta < C_\lambda$, we define:
			\begin{enumerate}[1.]
			\item Let $ S(\delta) > 0$ be the extended real number such that $\Gamma_\delta(s) : [0, S(\delta)) \to \mathbb{H} \times \R$ is the right maximally extended solution to the system \eqref{0316eq2.5}.
			\item For any positive integer $k$, let $s_{k}(\delta) > 0$ be the arc length parameter at the $k^{th}$ instance (if any) where $\sin\theta_\delta = 0$. If this never occurs, we set $s_{k}(\delta) = S(\delta)$. In particular, we put $s_{0}(\delta) =0$.
			\item Let $s_*(\delta) > 0$ be the arc length parameter at the first instance (if any) where $\dot{\theta}_\delta = 0$. If this never occurs, we set $s_*(\delta) = S(\delta)$.
		\end{enumerate}
	\end{defn}

Similarly, we can also define:
\begin{defn}\label{def3.2}
	For a real number $b$, we define:
	\begin{enumerate}[1.]
		\item Let $ \bar{S}(b) > 0$ be the extended real number such that $\overline{\Gamma}_b(s) : (0, \bar{S}(b)) \to \mathbb{H} \times \R$ is the maximally extended solution to the system \eqref{0316eq2.5}.
		\item For any positive integer $k$, let $\bar{s}_{k}(b) > 0$ be the arc length parameter at the $k^{th}$ instance (if any) where $\sin\bar{\theta}_b = 0$. If this never occurs, we take $\bar{s}_{k}(b) = \bar{S}(b)$. In particular, put $\bar{s}_{0}(b) =0$.
	\end{enumerate}
\end{defn}

Henceforth we assume $0<\delta < C_\lambda$.
when there is no risk of confusion, we may write $S(\delta)$, $s_{k}(\delta)$, $s_*(\delta)$, $\bar{S}(b)$ and $\bar{s}_{k}(b)$ as $S$, $s_k$, $s_*$, $\bar{S}$ and $\bar{s}_k$ respectively.
By definition, we have that $s_{0}< s_{1} \le s_{2}\le s_{3}\le ...\le S$ and that $s_{k-1}=s_{k}$ implies $s_{k-1}=s_{k}= s_{k+1}=...=S$ and that $s_{k-1}<s_{k}$ implies $s_{0}< s_{1} < s_{2}< s_{3}< ...< s_{k}$.
Moreover, if $s_{k-1}<s_{k}$, then for $s\in (s_{k-1},s_{k})$, we get $\sin \theta_\delta(s)>0$ provided $k$ is odd and $\sin \theta_\delta(s)<0$ provided $k$ is even.
A similar statement for $\bar{s}_k$, $\bar{S}$ also holds.

If $s_{k-1}<s_{k}$, then the curve $\gamma_\delta(s)$, $s_{k-1}<s<s_{k}$ can be written as a graph  over the $r$-axis.
We denote this graph by $(f_{k,\delta}(r), r)$.
If $k$ is odd, then the function $f_{k,\delta}(r)$ defined on $(r_\delta(s_{k-1}), r_\delta(s_{k}))$ is a maximally extended solution to (\ref{0526eq2.8}).
If $k$ is even, then the function $f_{k,\delta}(r)$ defined on $(r_\delta(s_{k}), r_\delta(s_{k-1}))$ is a maximally extended solution to (\ref{231212eq2.9}).
A similar statement for $\overline{\gamma}_b(s)$ also holds, and the corresponding function will be denoted by $\bar{f}_{k,b}$.
Without ambiguity, we may write $f_{k,\delta}$ and $\bar f_{k,b}$ as $f_{k}$ and $\bar f_k$ respectively.

Using definition \ref{def3.1}, we define the classes of $\gamma_\delta$ when $\delta$ in $(0,C_\lambda)$ and the classes of $\overline{\gamma}_b$ as follows:
\begin{defn}\label{231208def3.2}
	For $\delta \in (0, C_\lambda)$, $\gamma_\delta$ is said to belong to class $\mathcal{C}_k(N_1, \dots, N_k)$
	if $s_0<s_1<...<s_{k-1}<s_k$ and $f_1$ is of type $N_1$, $f_2$ is of type $N_2$,\,...\,,\,$f_{k}$ is of type $N_k$, where $N_i \in \left\lbrace1,2,3 \right\rbrace$, $i=1, ...,k$.
\end{defn}
\begin{defn}\label{240113def3.3}
    For a real number $b$, $\overline{\gamma}_b$ is said to to belong to class $\mathcal{C}_k(N_1, \dots, N_k)$ if $\bar{s}_0<\bar{s}_1<...<\bar{s}_{k-1}<\bar{s}_k$ and $\bar{f}_1$ is of type $N_1$, $\bar f_2$ is of type $N_2$, ... ,$\bar{f}_{k}$ is of type $N_k$, where $N_i \in \left\lbrace1,2,3 \right\rbrace$, $i=1, ...k$.
\end{defn}

 Using proposition \ref{231204lem3.3}, we have the following corollaries.
\begin{cor}\label{240114cor3.9}
	Let $\lambda<0$, $\delta\in(0,C_\lambda)$. Suppose $k>0$ is even with $s_{k-1}<s_k$.
	Then $s_k<S$ if and only if $f_k$ is of type 1 or type 2, $s_k=S$ if and only if $f_k$ is of type 3.
\end{cor}
\begin{cor}\label{240114cor3.10}
	Let $\lambda<0$, $b$ be a real number.  Suppose $k>0$ is even with $\bar s_{k-1}<\bar s_k$.
	Then $\bar s_k<\bar S$ if and only if $\bar f_k$ is of type 1 or type 2, $\bar s_k=\bar S$ if and only if $\bar f_k$ is of type 3.
\end{cor}


For the upcoming continuity argument (the shooting method) in section 7, it is crucial to show that these global classifications are robust. The following lemmas ensure that the class of a curve remains unchanged under small perturbations to the initial data.
\begin{lem}\label{231209lem3.3}
	Let $\delta_0 \in (0, C_\lambda)$ and $k \ge 1$. Suppose that $s_k(\delta_0) < S(\delta_0)$ and the curve $\gamma_{\delta_0}$ belongs to the class $\mathcal{C}_k(N_1, \dots, N_k)$. Then there exists an $\varepsilon > 0$ such that for all $\delta \in (\delta_0 - \varepsilon, \delta_0 + \varepsilon)$, we have $s_k(\delta) < S(\delta)$ and $\gamma_{\delta}$ also belongs to $\mathcal{C}_k(N_1, \dots, N_k)$.
\end{lem}
\begin{lem}\label{231214lem2}
	Let $b_0<-\lambda$ and $k \ge 1$. Suppose that $\bar s_k(b_0) < \bar S(b_0)$ and the curve $\bar \gamma_{b_0}$ belongs to the class $\mathcal{C}_k(N_1, \dots, N_k)$. Then there exists an $\varepsilon > 0$ such that for all $b \in (b_0 - \varepsilon, b_0 + \varepsilon)$, we have $\bar s_k(b) < \bar S(b)$ and $\bar \gamma_{b}$ also belongs to $\mathcal{C}_k(N_1, \dots, N_k)$.
\end{lem}

In \cite{LCW}, the authors also proved that:
\begin{lem}\label{231211lem3.4}
	Let $\delta_0$ be in $(0,C_\lambda)$ and $\gamma_{\delta_0}$ belong to \(\mathcal{C}_1(N)\) with \(N=1\) or \(2\).
	Then there is an $\varepsilon>0$ such that, for $\delta$ in $(\delta_0-\varepsilon,\delta_0+\varepsilon)$, $\gamma_\delta$ also belongs to \(\mathcal{C}_1(N)\).
\end{lem}
For the proof of these lemmas, we note that \(\ddot r\ne0\) whenever \(\dot r=0\), and that \(\ddot{\theta}\ne0\) whenever \(\dot \theta=0\). Otherwise \(r\) is constant.

\section{Behavior of the curve $\bar\gamma_b$}

In the previous section, we classified the general solutions. We now focus on a specific family of profile curves, denoted by $\bar{\gamma}_b$, which intersects the $x$-axis perpendicularly at a height $b$. Our goal in this section is to understand how the geometric shape and the class of $\bar{\gamma}_b$ depend on the initial height $b$.
We begin by noting that if $b<-\lambda$ then $\bar f'_{1,b}(r)>0$ and $\bar f''_{1,b}(r)>0$. If additionally $\lambda<0$ then by proposition \ref{231211lem3.3} we know that $\bar s_1(b) < \bar S(b)$.
\begin{lem}\label{231203lem4.1}
If $\lambda < 0$, $-\lambda- \frac{1}{2\sqrt{\pi}} \le b <-\lambda $, then $\bar r_b(\bar s_1) > \sqrt{\log \frac{-1}{2\sqrt{\pi}(b+\lambda)}}+\lambda$.
\end{lem}
\begin{proof}
	To avoid cumbersome notation, put $h(r)=\bar f_{1,b}(r)+\lambda$.
	Since $h'(0) = 0$, $h''(r) > 0$, and $\lim_{r \rightarrow \bar r_b(\bar s_1)}h'(r) = \infty$, there exists $r' \in (0, \bar r_b(\bar s_1))$ such that $h'(r') = 1$.
	For $r \in (0, r')$, we have
	\[
	\aligned
	\dfrac{{\rm d}}{{\rm d}r}(e^{-(r-\lambda)^2}h'(r))
	&= e^{-(r-\lambda)^2}h''(r) - 2(r-\lambda)e^{-(r-\lambda)^2}h'(r)\\
	&\le \dfrac{2}{1 + h'(r)^2} e^{-(r-\lambda)^2}h''(r) - 2(r-\lambda)e^{-(r-\lambda)^2}h'(r)\\
	&= 2e^{-(r-\lambda)^2}\left[ \left( r - \dfrac{n - 1}{r}\right) h'(r) - h(r) + \lambda(1 - \sqrt{1 + h'(r)^2}) \right]\\
	&\ \ \ \ - 2(r-\lambda)e^{-r^2}h'(r)\\
	&\le 2e^{-(r-\lambda)^2}\left[ \left( r - \dfrac{n - 1}{r}\right) h'(r) - h(r) - \lambda\,h'(r) \right] - 2(r-\lambda)e^{-r^2}h'(r)\\
	&\le -2e^{-(r-\lambda)^2}h(r),
	\endaligned
	\]
	where we have used the fact that $\lambda < 0$ in the penultimate inequality.
	Integrating from $0 \ {\rm to}\  r'$,
	\[
	e^{-(r'- \lambda)^2} \le -2\int_{0}^{r'}e^{-(r- \lambda)^2}h(r){\rm d}r \le -2(b+\lambda)\int_{0}^{r'}e^{-(r- \lambda)^2}{\rm d}r \le -2\sqrt{\pi}(b+\lambda).
	\]
	Hence,  we obtain $\bar r_b(\bar s_1) > r' \ge \sqrt{\log \frac{-1}{2\sqrt{\pi}(b+\lambda)}}+\lambda$ for $\lambda \le 0$, $-\lambda- \frac{1}{2\sqrt{\pi}} \le b <-\lambda$.
\end{proof}

\begin{lem}\label{231207lem4.2}
	Let $\lambda<0$ and $\max\left\lbrace 0, -\lambda - \dfrac{1}{2\sqrt{\pi}}e^{-(C_{-\lambda}+\sqrt{2}-\lambda)^2}\right\rbrace < b < -\lambda$. Then $\bar{\gamma}_b$ belongs to $\mathcal{C}_2(3,1)$.
\end{lem}
\begin{proof}
    We first note that $\bar{r}_b(\bar{s}_1)< \infty$ and $\bar{f}'_{2,b}(r)<0$ for $r$ close to $\bar{r}_b(\bar{s}_1)$.
	By the fact that $\lambda<0$ and corollary \ref{240113cor3.3}, we have $\sqrt{n-1}<C_{-\lambda}$ and $\bar{r}_b(\bar{s}_{2})<C_{-\lambda}$.
	It follows that both $\bar{r}_b(\bar{s}_1)-\sqrt{n-1}$ and $\bar{r}_b(\bar{s}_1)-\bar{r}_b(\bar{s}_2)$ are greater than $\sqrt{2}$ by lemma \ref{231203lem4.1} and the assumption on $b$.
	From $b>0$, we also have $\bar{f}_{2,b}(\bar{r}_b(s_1)) = \bar{f}_{1,b}(\bar{r}_b(s_1)) >0$.
	Now we can apply lemma \ref{231205lem3.4} to obtain that there exists a point $c_0$ in $[\bar{r}_b(\bar{s}_1)-\sqrt{2}, \bar{r}_b(\bar{s}_1))$ such that $\bar{f}_{2,b}'(c_0) = 0$. This completes the proof.
\end{proof}
Lemma \ref{231207lem4.2} guarantees the existence of class \(\mathcal{C}_2(3,1)\) curves for initial heights $b$ sufficiently close to $-\lambda$. By employing a continuity argument, the next lemma extends this geometric classification to a much wider range of initial heights.
\begin{lem}\label{240111lem4.3}
	Let $\lambda<0$ and $ -R_\lambda < b < -\lambda$. Then $\overline{\gamma}_b$ belongs to $\mathcal{C}_2(3,1)$, in particular, $\overline{\gamma}_0$ belongs to $\mathcal{C}_2(3,1)$.
\end{lem}
\begin{proof}
	Consider the set \[ 	\left\lbrace \tilde{b} \in (-\infty, -\lambda): \forall b \in (\tilde{b} ,-\lambda),\ \overline{\gamma}_b\text{ belongs to }\mathcal{C}_2(3,1) \right\rbrace, \] which is non-empty by lemma \ref{231207lem4.2}.
	Let $b_0$ be the infimum of this set: \[ b_0 = \inf \left\lbrace \tilde{b} \in (-\infty, -\lambda): \forall b \in (\tilde{b} ,-\lambda),\ \overline{\gamma}_b\text{ belongs to }\mathcal{C}_2(3,1) \right\rbrace. \]
	We will prove that $ b_0 = -R_{\lambda}. $
	First, we know that equation (\ref{0316eq2.5}) has a solution \[(-R_{\lambda}\cos\frac{s}{R_{\lambda}}, R_{\lambda}\sin\frac{s}{R_{\lambda}}, \frac{\pi}{2} -\frac{s}{R_{\lambda}} ),\ \  0<s<\pi R_\lambda,\] thus $\overline{\gamma}_{-R_\lambda}$ belongs to $\mathcal{C}_2(3,3)$.
	This yields that $b_0 \ge -R_\lambda.$
	Second, if $b_0 > -R_\lambda$, then by the definition of $b_0$, lemma \ref{231214lem2} and corollary \ref{240114cor3.10}, $\overline{\gamma}_{b_0}$ belongs to $\mathcal{C}_2(3,3)$ and therefore, using proposition \ref{231206lem3.5}, $\bar{\gamma}_{b_0}(-s)$ generates a convex, compact $(-\lambda)$-hypersurface.
	But Heilman's result \cite{Hei} tells us that this must be the sphere of radius $R_\lambda$,  which implies $b_0 = -R_\lambda$, contradicting the assumption that \(b_0>-R_\lambda\).
	This completes the proof of this lemma.
\end{proof}

According to the examples constructed in \cite{CLW}, the condition $\lambda<0$ in the above lemma is crucial.
In figure \ref{0426fig4.1}, we see that the red curves belong to \(\mathcal{C}_2(3,1)\).
\begin{figure}[htbp]
	\centering
	\includegraphics[scale=0.4]{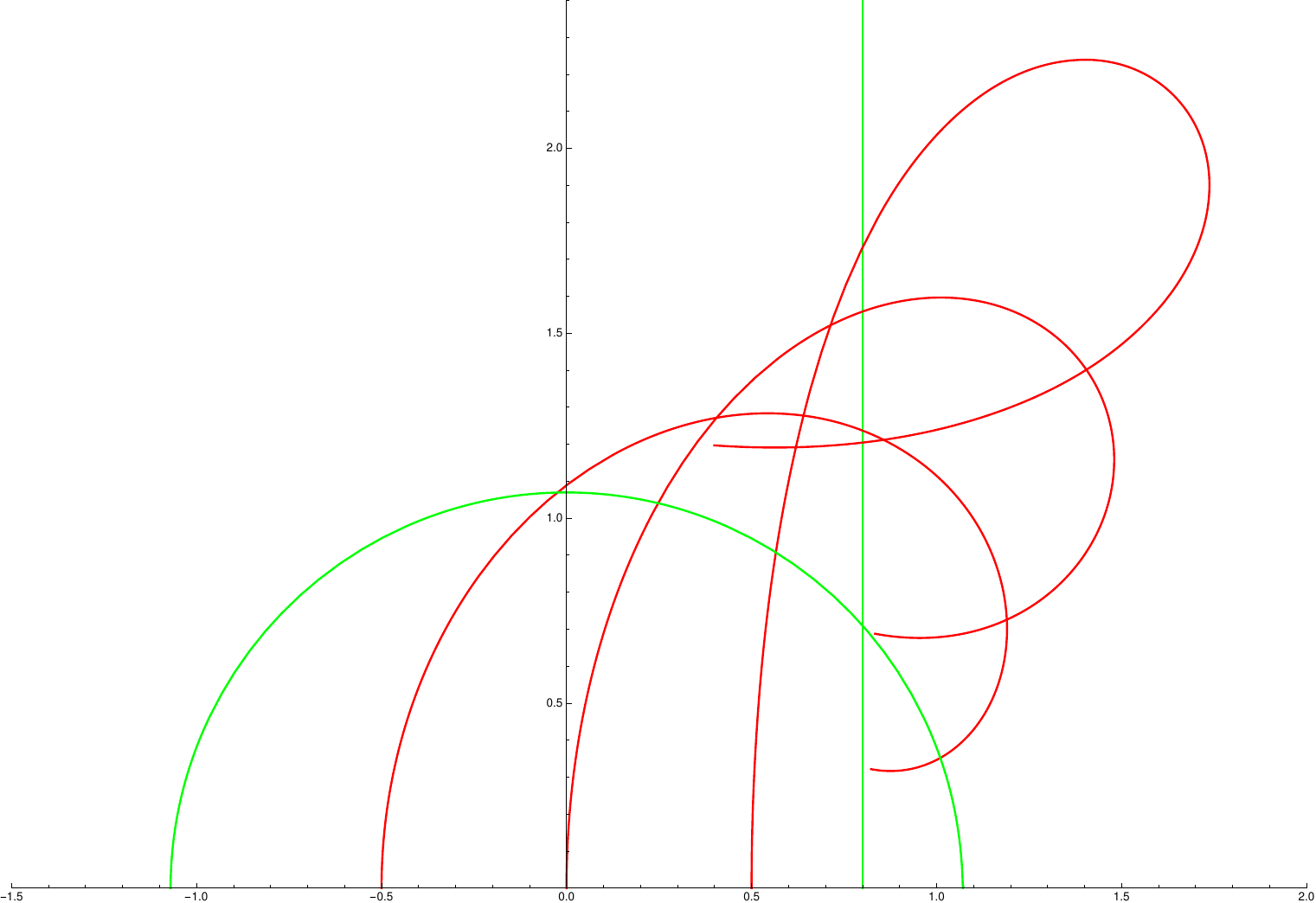}
	\caption{\ \ Graphs of some  $\overline{\gamma}_b$ with $ -R_\lambda < b < -\lambda$}
	\label{0426fig4.1}
\end{figure}

\section{Behavior of the curve $\gamma_\delta$}

Having analyzed the curves $\bar{\gamma}_b$ starting from the $x$-axis, we turn our attention to the one-parameter family of curves $\gamma_\delta$ emanating from the \(r\)-axis.
We begin with a lemma which can be found in \cite{LCW}.
\begin{lem}\label{lem5.1}
	For any $m \in \mathbb{N}$, there is a $T_m > 0$ and a $\delta_m > 0$ such that for all $0< \delta < \delta_m$, one has $T_m\,\delta < s_{*}(\delta)$, while, at $s = T_m\,\delta$, one will have $\arctan m < \theta_\delta < \pi/2 $, $x_\delta = O(\delta)$ and $r_\delta = \delta + O(\delta)$.
\end{lem}

Then the authors provide a remark and a proposition in the same paper \cite{LCW}.
\begin{rmk}\label{0529rmk3.1}
	For any fixed $m \in \mathbb{N}$, one can choose sufficiently small $\delta_m >0$ such that $x_\delta(T_m\,\delta) = O(\delta) <\frac{1}{m}$ and $r_\delta(T_m\,\delta) = \delta + O(\delta) <1$ for $0<\delta < \delta_m$.	
	Consider the tangent line $l_{m,\delta} : r - r_\delta(T_m\,\delta) = \tan \theta_\delta(T_m\,\delta)(x - x_\delta(T_m\,\delta))$ to $\gamma_\delta$ at $s = T_m\,\delta$. Setting $r = 1$ in the equation of $l_{m,\delta}$, we get $$ x = \frac{1}{\tan \theta_\delta(T_m\,\delta)}(1 - O(\delta) - \delta) + O(\delta) < \frac{2}{m} $$ for $0<\delta < \delta_m$.
	Therefore, if $0< \delta < \delta_m$ and the profile curve $\gamma_\delta(s) = (x_\delta, r_\delta),\ s\in (0, b]$ satisfies $\dot{r}_\delta > 0$, $r_\delta \le 1$ and $\dot{\theta}_\delta > 0$ on $(0, b]$, one will have $x_\delta \le \frac{2}{m} $ on $(0, b]$.
\end{rmk}

Henceforth we choose $\delta_m$ as in the above remark.
\begin{prop}\label{0529lem3.7}
	Let $\lambda < 0$. Then there is an $\hat{\delta} > 0$ such that for $0< \delta < \hat{\delta}$, $s_{*}(\delta) < s_1(\delta) $. In other words, the curve \(\gamma_\delta\) belongs to \(\mathcal{C}_1(2)\).
\end{prop}

This proposition and the following proposition \ref{240112prop5.2} characterize the behavior of the curve \(\gamma_\delta\) when \(\delta\) is sufficiently small. Furthermore, by applying proposition \ref{231211lem3.3}, we also deduce from this result that \(s_1(\delta) < S(\delta)\) whenever \(\lambda < 0\) and \(\delta\) is sufficiently small.

To prove proposition \ref{240112prop5.2}, we give a lemma.
\begin{lem}\label{lem5.2}
	Let $\lambda < 0$. Then we have $\lim_{\delta \to 0}r_\delta(s_{*}(\delta)) =0$, $\lim_{\delta \to 0}x_\delta(s_{*}(\delta))=0$ and $\lim_{\delta \to 0}\theta_\delta(s_{*}(\delta))=\pi/2$.
\end{lem}
\begin{proof}
	We argue by contradiction as in \cite{D}.
	Without loss of generality, assume $0< \delta < \hat{\delta}$.
	Suppose $\lim_{\delta \to 0}r_\delta(s_{*}(\delta)) \ne 0$. Then there is an $\varepsilon_1>0$ such that for all $m \in \mathbb{N}$, there exists $0< \bar{\delta}_m < \min\left\lbrace \frac{1}{m}, \delta_m\right\rbrace $ such that $r_{\bar{\delta}_m}(s_{*}(\bar{\delta}_m))\ge \varepsilon_1$.
	Putting $f_m(r) = f_{1,{\bar{\delta}_m}}(r)$, we have $f_m(r)>0$ and $f''_m(r)>0$ for $r$ in $(\bar{\delta}_m, \varepsilon_1)$.
	These will lead to a contradiction, hence $\lim_{\delta \to 0}r_\delta(s_{*}(\delta)) = 0$.

	Suppose $\lim_{\delta \to 0}x_\delta(s_{*}(\delta)) \ne 0$.
	Then there is an $\varepsilon_2>0$ such that for all $m \in \mathbb{N}$, there exists $0< \hat{\delta}_m < \min\left\lbrace \frac{1}{m}, \delta_m\right\rbrace $ such that $x_{\hat{\delta}_m}(s_{*}(\hat{\delta}_m))\ge \varepsilon_2$.
	The curve $\gamma_{\hat{\delta}_m}(s)$, $0<s<s_{*}(\hat{\delta}_m)$ can be written as a graph of a function over the $x$-axis, Denoting this function by $u_m(x)$ for $0<x<x_{\hat{\delta}_m}(s_{*})$, we see that $u_m$ satisfies (\ref{0528eq2.7}).
	Since $x_{\hat{\delta}_m}(s_{*})\ge \varepsilon_1$ and $u_m(x_{\hat{\delta}_m}(s_{*})) = r_{\hat{\delta}_m}(s_{*})$ approaches zero, we know that $u_m(x)$ converges to zero uniformly on compacta.
	Then one can find a sequence $\left\lbrace \xi_m \right\rbrace $ in a compact interval such that $$\lim_{m \to \infty}u_{m} (\xi_m) =0,\ \ \lim_{m \to \infty}u'_{m} (\xi_m) =0,\ \ \lim_{m \to \infty}u''_{m} (\xi_m) =0 ,$$ but these contradict (\ref{0528eq2.7}).
	Hence $\lim_{\delta \to 0}x_\delta(s_{*}(\delta))=0$.

	The last limit follows from lemma \ref{lem5.1}.
\end{proof}

\begin{prop}\label{240112prop5.2}
	Let $\lambda < 0$.	Then there exists an $\bar{\delta} > 0$ such that, for $0<\delta< \bar{\delta}$, the profile curve $\gamma_\delta$ belongs to $\mathcal{C}_2(2,1)$.
\end{prop}
\begin{proof}
	Let us consider the intersections between the curve $\gamma_\delta$ and the line $\{r=c\}$ for a positive constant $c < C_\lambda$. Suppose that $\sin \theta_\delta>0$, $\cos \theta_\delta>0$, $x_\delta>0$, and $\dot \theta_\delta<0$ hold at one of these points. It then follows from equation \eqref{0316eq2.5} that, at such points, $x_\delta$ possesses an upper bound depending on $c$.
	Together with the continuous dependence on the initial data, lemma \ref{lem5.2} shows that $\gamma_\delta$ converges to $\overline{\gamma}_0$ as $\delta \to 0$ (see figure \ref{240116fig5.1}).
	Since the profile curve $\overline{\gamma}_0$ belongs to $\mathcal{C}_2(3,1)$, we know that the profile curve ${\gamma}_\delta$ belongs to $\mathcal{C}_2(2,1)$ for $\delta$ close to zero.
\end{proof}

In figure \ref{240116fig5.1}, we display graphs of some profile curves with $n=2$ and $\lambda=-1$.
The black curve is $\overline{\gamma}_0$, the green curve is $\gamma_{0.02}$ , the yellow curve is $\gamma_{0.014}$, the pink curve is $\gamma_{0.009}$ and the red curve is $\gamma_{0.004}$.
\begin{figure}[htbp]
	\centering
	\includegraphics[scale=0.6]{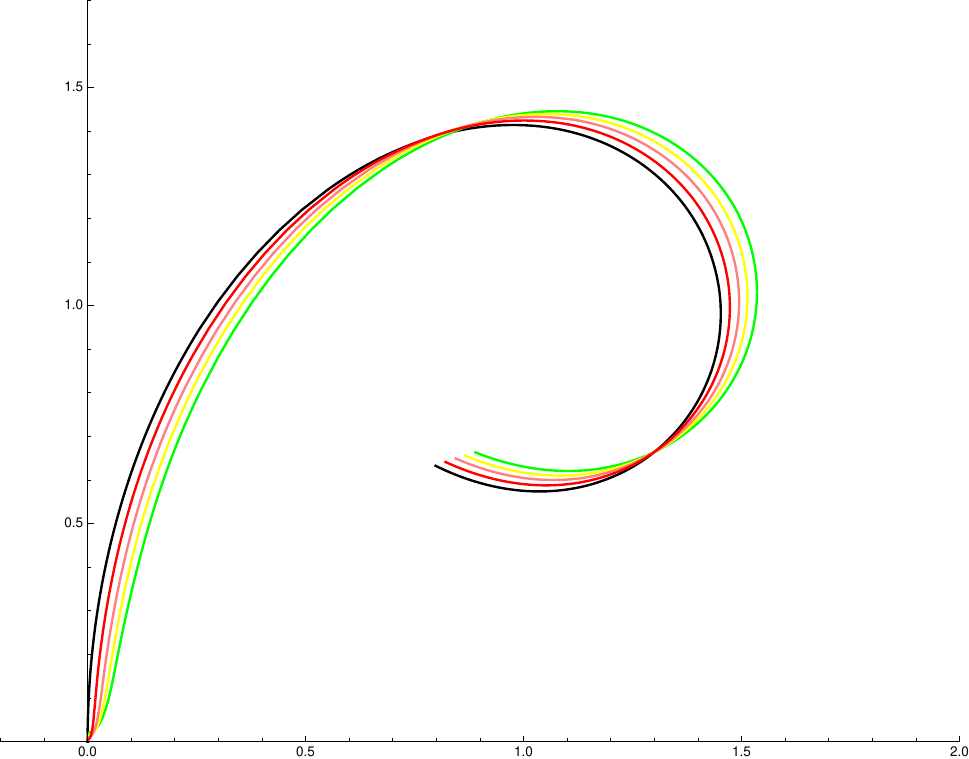}
	\caption{\ \ Graphs of some  ${\gamma}_\delta$ when $\delta$ is close to 0}
	\label{240116fig5.1}
\end{figure}

We also need to understand the behavior of the profile curves when the initial parameter $\delta$ approaches the critical upper bound $C_\lambda$. By analyzing the linearization around the constant solution, the following lemma establishes the oscillatory nature of the perturbed solution $u_\varepsilon(x)$, where $u_\varepsilon$ denotes the solution to equation (\ref{0528eq2.7}) with initial data $u_\varepsilon(0)=C_\lambda +\varepsilon$, $u_\varepsilon'(0)=0$.
\begin{lem}
	Given $n\ge 2$ and $\lambda\le-4\sqrt{\frac{n-1}{5}}$, there is an $\varepsilon_1>0$ such that $u_\varepsilon(x)=C_\lambda$ has at least three positive solutions for $|\varepsilon|<\varepsilon_1$.
\end{lem}
\begin{proof}
	We study the linearization of the equation (\ref{0528eq2.7}) near the constant solution $u_0(x)= C_\lambda$.
	Define $w$ by \[ w(x) = \left. \dfrac{{\rm d}}{{\rm d}\varepsilon}\right|  _{\varepsilon=0}u_\varepsilon(x). \]
	Since $u_\varepsilon$ satisfies equation \[ \dfrac{u_\varepsilon''}{1+(u_\varepsilon')^2} = x\,u_\varepsilon' - u_\varepsilon + \dfrac{n-1}{u_\varepsilon} + \lambda\sqrt{1+(u_\varepsilon')^2},\]
	by  differentiating  the above equation with respect to $\varepsilon$ and setting  $\varepsilon = 0$, we obtain a differential equation for $w$:
	\begin{equation}\label{eq5.1}
		 w'' - x\,w' + \left( 1+ \dfrac{n-1}{C_\lambda^2}\right)w =0.
	\end{equation} with initial conditions $w(0) = 1$ and $w'(0)=0$.
	The solution of this Cauchy problem is \[ _1F_1\left(-\dfrac{1}{2}\left( 1+ \dfrac{n-1}{C_\lambda^2}\right); \dfrac{1}{2}; \dfrac{x^2}{2} \right), \]
	where $_1F_1$ is the Kummer confluent hypergeometric function.
	The number of positive zeros of the above solution is \[ {\rm ceil}\left( \dfrac{1}{2}\left( 1+ \dfrac{n-1}{C_\lambda^2}\right)\right)  \ge 3,\] where ${\rm ceil}$ is the ceiling function.
	Here we have used the assumption on $\lambda$.
	We finish the proof by noting that the equation for $w$ has a unique constant solution $w=0$; therefore, by a uniqueness theorem of the solutions to (\ref{eq5.1}), $w' \ne 0$ whenever $w=0$, provided that $w$ is not the identically zero solution.
\end{proof}
This analytical property directly translates into the following geometric classification for $\gamma_\delta$.
\begin{cor}\label{240112cor5.4}
	Given $n\ge 2$ and $\lambda\le-4\sqrt{\frac{n-1}{5}}$, there is an $\varepsilon_1>0$ such that, for $\delta$ in $(C_\lambda - \varepsilon_1,C_\lambda)$, the curve $\gamma_\delta$ belongs to $\mathcal{C}_2(2,2)$.
\end{cor}

In figure \ref{240116fig5.2}, we display graphs of some profile curves with $n=2$ and $\lambda=-\sqrt{5}$, where the black curve is $\gamma_{C_\lambda}$, namely, $\gamma_{C_\lambda}(s) = (s, \frac{3-\sqrt{5}}{2})$.
\begin{figure}[htbp]
	\centering
	\includegraphics[scale=0.8]{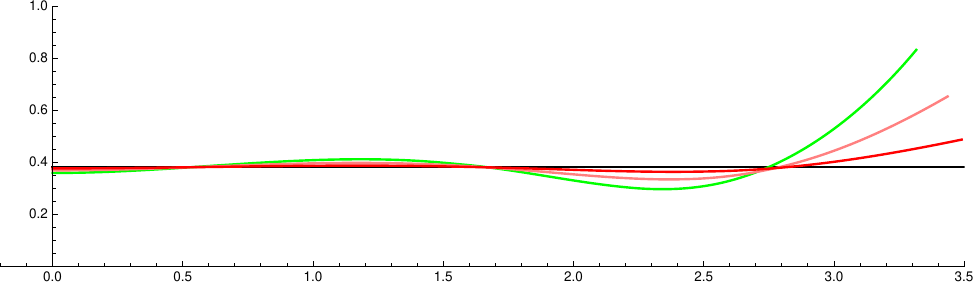}
	\caption{\ \ Graphs of some  ${\gamma}_\delta$ when $\delta$ is close to $C_\lambda$}
	\label{240116fig5.2}
\end{figure}

\section{Proof of the theorem}
We are now in a position to prove our main result. The strategy relies on a continuity method (shooting method).
By continuously varying the initial parameter $\tilde{\delta}$ and defining a critical infimum value $\delta_s$, we will piece together the results established in sections 3, 4, and 5. Our goal is to show that the limiting curve $\gamma_{\delta_s}$ does not exhibit degenerate behaviors.

\noindent {\it Proof of theorem \ref{0316thm1.1}.}
Let $\tilde{\lambda}=-\lambda \le-4\sqrt{\frac{n-1}{5}}$.
We first construct a $\tilde{\lambda}$-hypersurface.
Consider the set \[ 	\left\lbrace \tilde{\delta} \in (0, C_{\tilde{\lambda}}): \forall \delta \in (\tilde{\delta} ,C_{\tilde{\lambda}}),\ \gamma_\delta\ \text{belongs to } \mathcal{C}_2(2,2) \right\rbrace, \]
which is non-empty by corollary  \ref{240112cor5.4}.
Let $\delta_s$ be the infimum of this set: \[ \delta_s =  \inf	\left\lbrace \tilde{\delta} \in (0, C_{\tilde{\lambda}}): \forall \delta \in (\tilde{\delta} ,C_{\tilde{\lambda}}),\ \gamma_\delta\ \text{belongs to } \mathcal{C}_2(2,2) \right\rbrace, \]
proposition \ref{240112prop5.2} tells us that $\delta_s>0$.
Using the definition of $\delta_s$, continuous dependence on the initial data, proposition \ref{240112lem3.3} and lemma \ref{231205lem3.4}, one can prove that the curve $\gamma_{\delta_s}$ does not belong to \(\mathcal{C}_1(3)\).
By the definition of $\delta_s$, lemma \ref{231211lem3.4} shows that the curve $\gamma_{\delta_s}$ does not belong to $\mathcal{C}_1(1)$.
Hence the curve $\gamma_{\delta_s}$ belongs to \(\mathcal{C}_1(2)\), and then $s_1(\delta_s)<S(\delta_s)$ by proposition \ref{231211lem3.3}.
Note that \(s_1(\delta_s)<S(\delta_s)\) means that the curve $\gamma_{\delta_s}$ belongs to some \(\mathcal{C}_2(N_1,N_2)\).
Using lemma \ref{231209lem3.3}, corollary \ref{240114cor3.9} and the definition of $\delta_s$, the curve $\gamma_{\delta_s}$ belongs to \(\mathcal{C}_2(2,3)\) (see the red curve in figure \ref{240116fig6.1}).
Therefore, by the symmetry of the equation (\ref{0316eq2.5}) and proposition \ref{231206lem3.5}, the profile curve $\gamma_{\delta_s}(s)$, $-S(\delta_s)\le s \le S(\delta_s)$ generates an embedded, compact and non-convex $\tilde{\lambda}$-hypersurface.
One can readily verify from \eqref{0316eq2.2} that the corresponding unit normal vector points outward.
Switching  the unit normal vector to its opposite, we obtain  an embedded, compact and non-convex $\lambda$-hypersurface with an inward-pointing unit normal vector.
A non-convex \(\lambda\)-sphere in \(\mathbb{R}^3\) with \(\lambda=\sqrt{5}\) is shown in figure \ref{fig:4.3}.
\begin{figure}[htbp]
	\centering
	\includegraphics[scale=0.6]{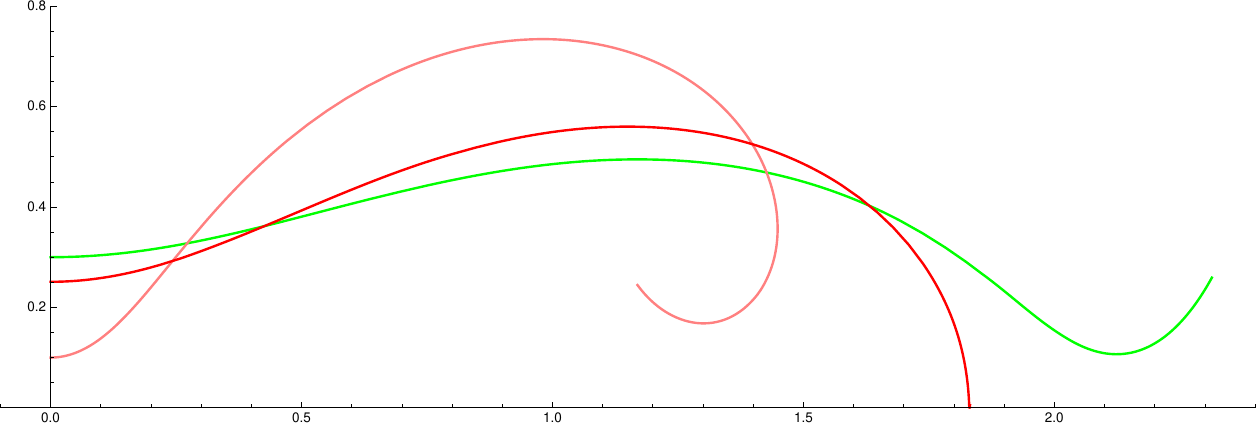}
	\caption{\ \ A curve which generates a mean convex $\mathbb S^n$ $\lambda$-hypersurface}
	\label{240116fig6.1}
\end{figure}

\begin{figure}[H]
\centering
\subfigure[] {	
	\includegraphics[width=0.4\columnwidth]{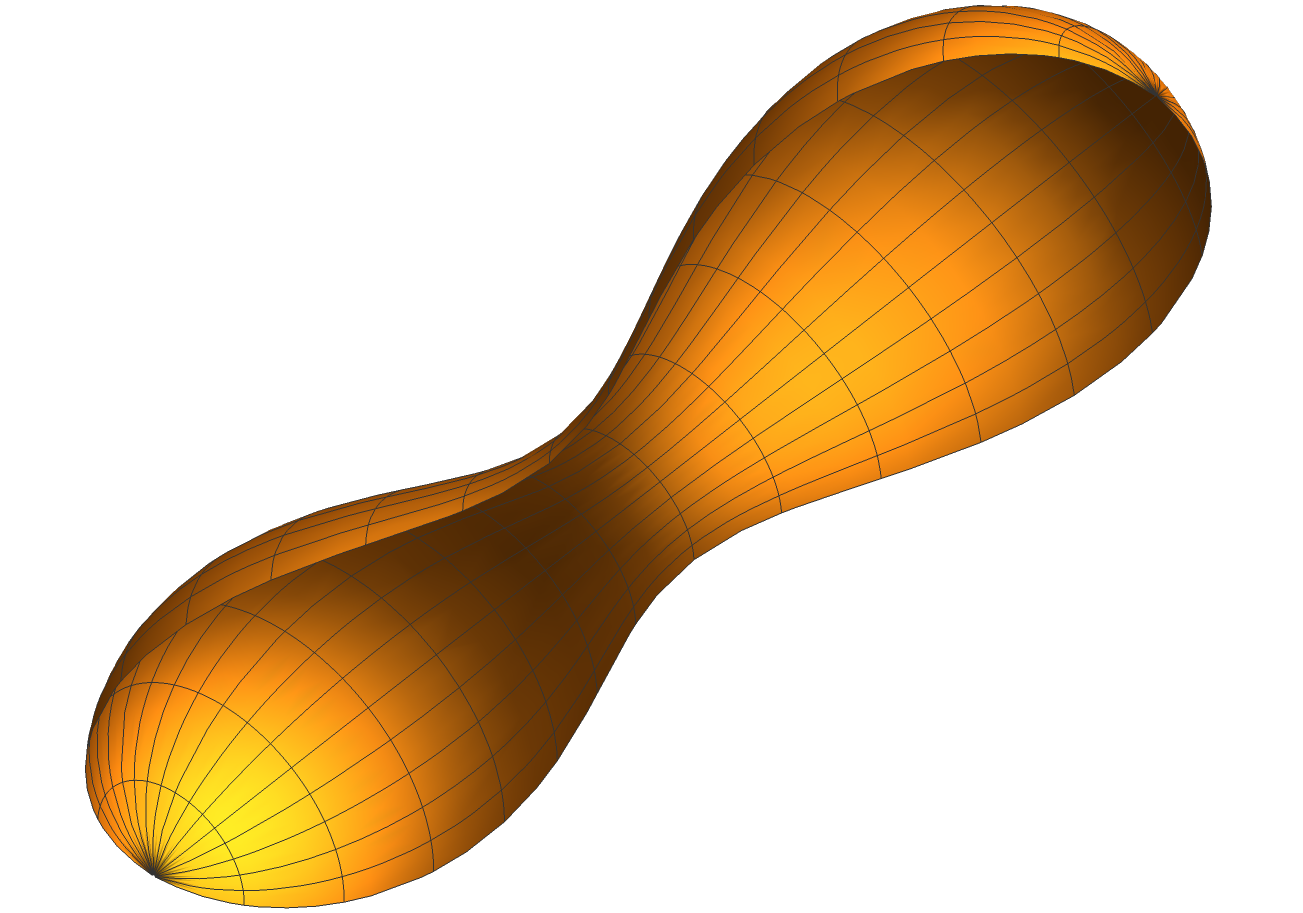}   \label{fig:4.3a}
}\hfil
\subfigure[] {	
	\includegraphics[width=0.36\columnwidth]{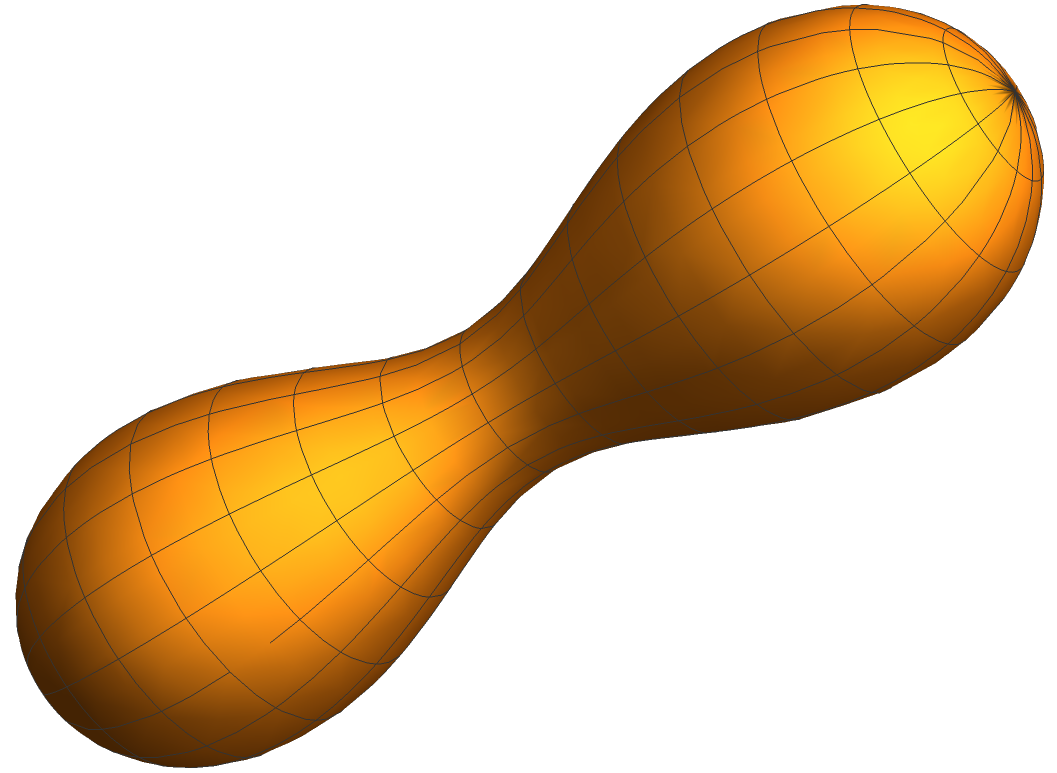}   \label{fig:4.3b}
}
\caption{\ A mean convex $\lambda$-sphere in $\R^3$ where $\lambda=\sqrt 5$. }   \label{fig:4.3}

\end{figure}

 We next show that the hypersurface $\Sigma$ is strictly mean convex. Since the hypersurface $\Sigma$ is compact, there exists a point $q\in \Sigma$ such that $H(q)=\min_{p\in \Sigma}H$, that is, $\< X, \nu\>(q)=\sup_{p\in \Sigma}\< X, \nu\>(p)$.
Then we have
\begin{equation}\label{eq:7-8-1}
\nabla_{e_i}\< X, \nu\>(q)=\<e_i, \nu\>+\<X, -\kappa_i e_i\>=-\kappa_i\<X, e_i\>=0, \ \ \ \ i=1, 2, \cdots,n,
\end{equation}
where $\nabla$ is the gradient operator on $\Sigma$ and $\{e_1, \cdots, e_n\}$  is a local orthonormal frame on $\Sigma$. From equation (2.3), we get $\kappa_j=-\frac{\dot{x}}{r}>0$, $j=1,2,\cdots,n-1$. If $H(q)=\min_{p\in \Sigma}H\leq0$, that is $\< X, \nu\>(q)=\lambda-H(q)>0$, then we obtain $\kappa_n=H-(\kappa_1+\kappa_2+\cdots+\kappa_{n-1})<0$.
It follows from \eqref{eq:7-8-1} that
\begin{equation}
\<X, e_i\>(q)=0, \ \ \ \ i=1, 2, \cdots,n.
\end{equation}
That is, $X(q)\parallel \nu(q)$. $\nu$ is the inward unit normal vector, we conclude $\< X, \nu\>(q)<0$, it is a contradiction.
Hence,  $H(q)=\min_{p\in \Sigma}H>0$.
Therefore, $\Sigma$ is strictly mean convex.\\



\begin{thebibliography}{1}
		
		\bibitem{A} S. B. Angenent, \emph{Shrinking doughnuts}, Birkh\"auser, Boston-Basel-Berlin,  {\bf 7}, 21-38, 1992.
		\bibitem{B} S. Brendle,  \emph{Embedded self-similar shrinkers of genus 0}, Ann. of Math. {\bf 183} (2016), 715-728.
		\bibitem{C} J. -E. Chang,  \emph{1-dimensional solutions of the $\lambda$-self shrinkers}, Geom. Dedicata  {\bf 189} (2017), 97-112.
		\bibitem{CM} T. H. Colding and W. P. Minicozzi II,  \emph{Generic mean curvature flow I; generic singularities}, Ann. of Math. {\bf 175} (2012), 755-833.
		\bibitem{CS} A. C. -P. Chu and A. Sun, \emph{Genus one singularities in the mean curvature flow}, arXiv: 2308.05923.
		\bibitem{CW}  Q. -M. Cheng and G. Wei,  \emph {Complete $\lambda$-hypersurfaces of weighted volume-preserving mean curvature flow},
		Cal. Var. Partial Differential Equations {\bf57} (2018), 32.
		\bibitem{CW1} Q. -M. Cheng and G. Wei, \emph{Examples of compact $\lambda$-hypersurfaces in Euclidean spaces},
		Sci. China Math.  {\bf64} (2021), 155-166.
		\bibitem{CLW} Q. -M. Cheng, J. Lai and G. Wei, \emph{Examples of compact embedded convex $\lambda$-hypersurfaces}, J. Funct. Anal. {\bf286} (2024), no. 2, Paper No. 110211.
        \bibitem{LCW} Q. -M. Cheng, J. Lai and G. Wei, \emph{Embedded cylindrical and doughnut-shaped $\lambda$-hypersurfaces}, arXiv:2406.11123.
		\bibitem{DD} M. do Carmo and M. Dajczer, \emph{Hypersurfaces in space of constant curvature},
		Trans. Amer. Math. Soc. {\bf 277}  (1983),  685-709.
		\bibitem{D} G. Drugan, \emph{An immersed $S^2$ self-shrinker}. Trans. Amer. Math. Soc. {\bf367} (2015), 3139-3159.
		\bibitem{DK} G. Drugan and S. J. Kleene, \emph{Immersed self-shrinkers}, Trans. Amer. Math. Soc. {\bf 369} (2017), 7213-7250.
		\bibitem{G} Q. Guang, \emph{A note on mean convex $\lambda$-surfaces in $\R^3$}, Proc. Amer. Math. Soc. {\bf 149} (2021), 1259-1266.

        \bibitem{H} G. Huisken, \emph{Asymptotic behavior for singularities of the mean curvature flow}, J. Differential Geom. {\bf 31} (1990), 285-299.


		\bibitem{Hei} S. Heilman, \emph{Symmetric convex sets with minimal Gaussian surface area}, Amer. J. Math. {\bf 143} (2021), 53-94.	

		\bibitem{Hs5} W. Hsiang and W.-Y. Hsiang, \emph{On the existence of codimension-one minimal spheres in compact symmetric spaces of rank $2$. II}, J. Differential Geometry {\bf 17} (1982), no.~4, 583--594 (1983); MR0683166

        \bibitem{I} T. Ilmanen, \emph{Problems in mean curvature flow}, available at \nolinkurl{http://people.math.ethz.ch/\textasciitilde ilmanen/classes/eil03/problems03.pdf}.


		\bibitem{KM} S. Kleene and N. M. M\o ller, \emph{Self-shrinkers with a rotational symmetry}, Trans. Amer. Math. Soc. {\bf 366} (2014), no. 8, 3943-3963.
		\bibitem{L} T. -K. Lee,  \emph{Convexity of $\lambda$-hypersurfaces}, Proc. Amer. Math. Soc. {\bf 150}  (2022), 1735-1744.
		
		\bibitem{LW} Z. Li and G. Wei, \emph{An immersed $S^n$ $\lambda$-hypersurface}, J. Geom. Anal. {\bf 33} (2023), Paper No. 288, 29 pp.
		\bibitem{Liang} J. Liang, \emph{A singular initial value problem and self-similar solutions of a nonlinear dissipative wave equation}, J. Differential Equations {\bf 246} (2009), no.~2, 819--844; MR2468737
		\bibitem{M} P. McGrath, \emph{Closed mean curvature self-shrinking surfaces of generalized rotational type}, arXiv: 1507.00681.
		\bibitem{MR}  M. McGonagle and J. Ross, \emph{The hyperplane is the only stable, smooth solution to the isoperimetric problem in Gaussian space}, Geom. Dedicata {\bf178} (2015), 277-296.
		\bibitem{R} J. Ross, \emph{On the existence of a closed, embedded, rotational $\lambda$-hypersurface}, J. Geom.  {\bf110} (2019), 1-12.
		\bibitem{R1} O. Riedler, \emph{Closed embedded self-shrinkers of mean curvature flow}, J. Geom. Anal. {\bf33} (2023), 172.
	    \bibitem{S} A. Sun, \emph{Compactness and rigidity of $\lambda$-surfaces}, Int. Math. Res. Not. IMRN  (2021),  11818-11844.
		
		
		
		
	\end{thebibliography}
\end{document}